\documentclass[english]{amsart}

\usepackage{esint}
\usepackage[svgnames]{xcolor} 
\usepackage[colorlinks,citecolor=red,pagebackref,hypertexnames=false,breaklinks]{hyperref}
\usepackage{pgf,tikz}
\usepackage{pdfsync}

\usepackage{dsfont}
\usepackage{url}
\usepackage[utf8]{inputenc}
\usepackage[T1]{fontenc}
\usepackage{lmodern}
\usepackage{babel}
\usepackage{mathtools}  
\usepackage{amssymb}
\usepackage{lipsum}
\usepackage{mathrsfs}
\usepackage{color}
\usepackage[skip=2.2pt plus 1pt, indent=12pt]{parskip}
\usepackage{stmaryrd}
\usepackage{soul}

\newtheorem{theorem}{Theorem}[section]
\newtheorem{proposition}{Proposition}[section]
\newtheorem{lemma}{Lemma}[section]

\newtheorem{corollary}{Corollary}[section]

\newtheorem{remark}{Remark}[section]

\numberwithin{equation}{section}

\title[Inverse spectral problems for higher-order coefficients]{Inverse spectral problems for higher-order coefficients}

\author{Mourad Choulli}
\address{Universit\'{e} de Lorraine, 34 cours L\'{e}opold, 54052 Nancy cedex, France}
\email{mourad.choulli@univ-lorraine.fr}

\date{}

\begin{document}
\begin{abstract}
We establish uniqueness and stability inequalities for the problem of determining the higher-order coefficients of an elliptic operator from the corresponding boundary spectral data (BSD). Our analysis relies on the relationship between boundary spectral data and elliptic and hyperbolic Dirichlet to Neumann (DtN) maps. We also show how to adapt our analysis to obtain uniqueness and stability inequalities for determining the conductivity or the potential in an elliptic operator from the corresponding BSD.

\end{abstract}



\maketitle

\tableofcontents

\section{Introduction}\label{s1}

Let $n\ge 3$ be an integer and $\Omega$ be a $C^\infty$ bounded domain of $\mathbb{R}^n$, whose boundary is denoted $\Gamma$.

For simplicity, we use the Einstein summation convention for subscript quantities. This means that a term appearing twice, as a superscript and a subscript, is assumed to be summed from 1 to n.

We denote by $M_n^s$ the space of symmetric matrices $n\times n$. For fixed $\alpha >1$ and $\beta >0$, let $\mathbf{G}$ be the set of matrix-valued functions $g=(g_{k\ell})\in C^\infty(\overline{\Omega}, M_n^s)$ satisfying
\[
\alpha^{-1} |\xi|^2\le g^{k\ell}\xi_k\overline{\xi_\ell}\le \alpha |\xi|^2,\; \xi \in \mathbb{C}^n,\quad \max_{k,\ell}\|g^{k\ell}\|_{C^{0,1}(\overline{\Omega})}\le \beta ,
\]
where $(g^{k\ell}(\cdot)):=g(\cdot)^{-1}$.

We need to recall some usual definitions of Riemannian geometry. The Laplace-Beltrami operator associated with the metric tensor $g_{k\ell}dx^k\otimes dx^\ell$ is given by
\[
\Delta_g u:=\frac{1}{\sqrt{|g|}}\partial_k\left(\sqrt{|g|}g^{k\ell}\partial_\ell u\right),
\]
where $|g|:=\mbox{det} (g)$. 

Unless otherwise stated, the functions we will consider are complex-valued. 

The following notations will also be used
\begin{align*}
&\langle X,Y\rangle_g=g_{k\ell}X^kY^\ell,\quad X=X^k\frac{\partial}{\partial x^k},\; Y=Y^k\frac{\partial}{\partial x^k},
\\
&\nabla_gw=g^{k\ell}\partial_k w\frac{\partial}{\partial x^\ell},
\\
&|\nabla_gw|_g^2=\langle\nabla_gw,\nabla_g\overline{w}\rangle_g=g^{k\ell}\partial_k w\partial_\ell \overline{w}.
\end{align*}

Let $dV_g=\sqrt{|g|}dx$, where $dx$ is the Lebesgue measure on $\mathbb{R}^n$. The natural norms on $L^2(\Omega,dV_g)$ and $H^1(\Omega,dV_g)$ are respectively given by
\begin{align*}
&\|w\|_{0,2,g}:=\left(\int_\Omega |w|^2dV_g\right)^{\frac{1}{2}},
\\
&\|w\|_{1,2,g}:= \left(\int_\Omega \left[|w|^2+|\nabla_gw|_g^2\right]dV_g\right)^{\frac{1}{2}}.
\end{align*}

From now on, $\mathbf{e}=(\delta_{k\ell})$ will denote the identity matrix of $\mathbb{R}^n$. Here $\delta_{k\ell}$ is the usual Kronecker's symbol. That is $\delta_{kk}=1$ and $\delta_{k\ell}=0$ when $k\ne \ell$. In the above norms, we will remove the subscript $g$ when $g=\mathbf{e}$. We verify that the following  inequalities hold.
\begin{align*}
& \alpha^{-\frac{n}{4}} \|w\|_{0,2}\le \|w\|_{0,2,g} \le \alpha^{\frac{n}{4}}\|w\|_{0,2},
\\
& \alpha^{-\frac{n+2}{4}}\|\nabla w\|_{0,2}\le \|\nabla_g w\|_{0,2,g} \le \alpha^{\frac{n+2}{4}}\|\nabla w\|_{0,2} .
\end{align*}

Let $dS_g=\sqrt{|g|}dS$, where $dS$ is the surface measure on $\Gamma$. Unless there is any confusion, the norm of $L^2(\Gamma,dS_g)$ will also be denoted $\|\cdot\|_{0,2,g}$. The scalar products on $L^2(\Omega,dV_g)$ and $L^2(\Gamma,dS_g)$ are denoted respectively as follows.
\begin{align*}
&(u_1|u_2)_g=\int_\Omega u_1\overline{u_2}dV_g,
\\
&\langle f_1|f_2\rangle_g=\int_\Gamma f_1\overline{f_2}dS_g.
\end{align*}

For $g\in \mathbf{G}$, define the unbounded operator $A^g$ by
\[
A^gu=-\Delta_gu,\quad u\in D(A^g)=H_0^1(\Omega)\cap H^2(\Omega).
\]

The sequence of eigenvalues of $A^g$ will be denoted $(\lambda_k^g)$:
\[
0<\lambda_1^g\le \lambda_2^g\le \ldots \le \lambda_k^g\le \ldots, \quad \lim_{k\rightarrow \infty}\lambda_k^g=\infty.
\]
Let $(\phi_k^g)$ be an orthonormal basis of $L^2(\Omega,dV_g)$ consisting of eigenfunctions with
\[
\phi_k^g\in H_0^1(\Omega)\cap H^2(\Omega),\quad A^g\phi_k^g=\lambda_k^g\phi_k^g,\quad k\ge 1. 
\]
The following notation will be used hereinafter
\[
\psi_k^g=\partial_{\nu_g} \phi_k^g:= \langle \nabla_g\phi_k^g,\nu\rangle_g,\quad k\ge 1.
\]
where $\nu$ denotes the outward normal vector field on $\Gamma$ normalized by $\langle \nu,\nu\rangle_g=1$. 

We need  additional definitions. We recall that $\ell^r$, $1\le r<\infty$, denotes the space of complex sequences $(\xi_k)$ such that  $\sum_{k\ge 1}|\xi_k|^r<\infty$. $\ell^r$ is a Banach space when equipped with the norm
\[
\|(\xi_k)\|_{\ell^p}:=\left(\sum_{k\ge 1}|\xi_k|^r\right)^{\frac{1}{r}},\quad (\xi_k)\in \ell^r. 
\]

If $B$ is a Banach space with norm $\|\cdot \|_B$, $\ell^r(B)$ will denote the Banach space of sequences $(w_k)$ in $B$, verifying $\sum_{k\ge 1}\|w_k\|_B^r<\infty$. We equip $\ell^r(B)$ with the norm
\[
\|(w_k)\|_{\ell^r(B)}:=\|(\|w_k\|_B)\|_{\ell^r},\quad (w_k)\in \ell^r(B).
\]

The following property will be useful in the remaining part of this text: if $1\le r_1\le r_2$, then $\ell^{r_1}\subset \ell^{r_2}$ with $\|\cdot \|_{\ell^{r_2}}\le \|\cdot \|_{\ell^{r_1}}$. Therefore, we have also $\ell^{r_1}(B)\subset \ell^{r_2}(B)$ with $\|\cdot \|_{\ell^{r_2}(B)}\le \|\cdot \|_{\ell^{r_1}(B)}$.

Also, $\ell ^1(\mathbb{N}^2)$ will be the Banach space of complex sequences $(\xi_{k\ell})$ satisfying $\sum_{k,\ell \ge 1}|\xi_{k\ell}|<\infty$. The norm on $\ell ^1(\mathbb{N}^2)$ is given by
\[
\|(\xi_{k\ell})\|_{\ell ^1(\mathbb{N}^2)}=\sum_{k,\ell \ge 1}|\xi_{k\ell}|.
\]

Fix $g_0\in \mathbf{G}$ and set
\[
\mathbf{G}_0=\{ g\in \mathbf{G};\; g_{|\Gamma}=g_0\}.
\]

To simplify the notations, any quantity of the form $X_k^{g_j}$ will be written $X_k^j$, $j=1,2$, $k\ge 1$.

Fix $1\le p<\frac{2n}{2n-1}$ and $1\le q<\frac{4n}{4n-1}$ and, for $g_1,g_2\in \mathbf{G}_0$, define
\begin{align*}
&\delta(g_1,g_2):=\|(\lambda_k^1-\lambda_k^2)\|_{\ell^p}+\|(\psi_k^1-\psi_k^2)\|_{\ell^q(L^2(\Gamma))},
\\
&\delta_0(g_1,g_2):=\|(\ell^{\frac{7}{4n}}k^{-\frac{1}{4n}}(\delta_{k\ell}-(\phi_k^2|\phi_\ell^1)_{g_1}))\|_{\ell^1(\mathbb{N}^2)},
\\
&\delta_+(g_1,g_2):= \delta(g_1,g_2)+\delta_0(g_1,g_2).
\end{align*}

Finally, we define the extension operators that will be used in the rest of this text. For $\varphi\in  H^{\frac{3}{2}}(\Gamma)$, let $E\varphi$ be the unique element of $H^2(\Omega)$ satisfying
\[
\|E\varphi\|_{H^2(\Omega)}=\|\varphi\|_{H^{\frac{3}{2}}(\Gamma)}.
\]
Here and henceforth,  $H^{\frac{3}{2}}(\Gamma)$ will be endowed with the quotient norm
\[
\|\varphi\|_{H^{\frac{3}{2}}(\Gamma)}=\min\{\|\Phi\|_{H^2(\Omega)};\; \Phi \in H^2(\Omega),\; \Phi_{|\Gamma}=\varphi\}.
\]

Let $\tau >0$ and $j\ge 1$. For $h\in H^j((0,\tau),H^{\frac{3}{2}}(\Gamma))$, we denote again $Eh$ the extension of $h$ given by 
\[
Eh (\cdot,t)=E(h(\cdot,t)),\quad t\in [0,\tau], 
\]
in such a way that $Eh\in H^j((0,\tau), H^2(\Omega))$.

Unless otherwise specified, the constant $\mathbf{c}>0$ will represent a generic constant depending only on $n$, $\Omega$, $\alpha$, and $\beta$.

\subsection{The conformal factor}\label{sb1.1}

Let $\Omega'$ be a bounded domain of $\mathbb{R}^{n-1}$ of class $C^\infty$ such that $\overline{\Omega} \subset \Omega' \times \mathbb{R}$, and let $g' \in C^\infty(\overline{\Omega'}, M_{n-1}^s)$. For $M':=\overline{\Omega'}$, assume that $(M', g')$ is a simple manifold, that is, for every $x \in M'$, the exponential map, denoted by $\mathrm{exp}_x$, is a diffeomorphism from its maximal domain in $T_xM'$ onto $M'$, and that the boundary $\partial M'$ is strictly convex with respect to the metric $g'$.

Let $\tilde{g}\in C^\infty (\overline{\Omega},M_n^s)$ given
\[
(\tilde{g})_{k\ell}=(g')_{k\ell},\; 1\le k,\ell\le n-1,\quad (\tilde{g})_{kn}=\delta_{kn},\; 1\le k\le n.
\]

In the following, under the conditions above, we will say that $(\Omega, \tilde{g})$ is admissible.

Define
\begin{align*}
&\mathbf{C}=\{c\in C^\infty(\overline{\Omega});\; \alpha^{-1}\le c\le \alpha,\; \|c\|_{C^3(\overline{\Omega})}\le \tilde{\beta} \},
\\
&\mathbf{G}^{\rm{c}}=\{g=c \tilde{g};\; c\in \mathbf{C}\},
\end{align*}
and for fixed $g_0\in \mathbf{G}^{\rm{c}}$,  set
\[
\mathbf{G}_0^{\rm{c}}=\{g\in \mathbf{G}^{\rm{c}};\; g_{|\Gamma}=g_0{_{|\Gamma}}\}.
\]
Clearly, $\mathbf{G}^{\rm{c}}\subset \mathbf{G}$ and $\mathbf{G}_0^{\rm{c}}\subset \mathbf{G}_0$, for some $\beta$ depending only on $\tilde{\beta}$ and $\|\tilde{g}\|_{C^3(\overline{\Omega})}$.

Also, define
\[
\Psi_{\varsigma,\theta}=|\ln(\delta+| \ln \delta|^{-1}) |^{-\theta}\chi_{(0,\varsigma]}+\delta \chi_{(\varsigma,\infty)},\quad 0<\varsigma<e^{-1},\; 0<\theta <1.
\]
We extend $\Psi_{\varsigma,\theta}$ by continuity at $\delta=0$ by setting $\Psi_{\varsigma,\theta}(0)=0$. Here, the notation $\chi_I$ refers to the characteristic function of the interval $I$.

\begin{theorem}\label{thm2.1}
There exist three constants $\kappa>0$, $0<\varsigma<e^{-1}$ and $0<\theta <1$,  depending only on $n$, $\Omega$, $\tilde{g}$, $\alpha$, $\tilde{\beta}$ such that for all $g_1=c_1\tilde{g}\in \mathbf{G}_0^{\rm{c}}$ and $g_2=c_2\tilde{g}\in \mathbf{G}_0^{\rm{c}}$ satisfying $\delta_+(g_1,g_2)<\infty$ we have
\[
\|c_1-c_2\|_\infty \le \kappa \Psi_{\varsigma,\theta}\left(\delta (g_1,g_2)\right),
\] 
where $\|\cdot\|_\infty$ represents the $L^\infty$ norm.
\end{theorem}

\subsection{The metric}\label{sb1.2}

First, we recall a stability inequality concerning the determination of the metric from the hyperbolic Dirichlet to Neumann map.

Let $\tilde{\mathbf{G}}$ denotes the set of functions $g\in C^\infty(\overline{\Omega},M_n^s)$ such that $g(x)$ is positive definite for all $x\in \overline{\Omega}$. For fixed $\tau >0$, set $Q=\Omega \times (0,\tau)$ and $\Sigma=\Gamma \times (0,\tau)$. Define 
\[H_{0,}^1(\overline{\Sigma})=\{h\in H^1(\overline{\Sigma});\; h(\cdot ,0)=0\}. 
\]
According to \cite[Theorem 2.30]{KKL} for all $h\in H_{0,}^1(\overline{\Sigma})$ the IBVP
\[
\left\{
\begin{array}{ll}
(\partial_t^2-\Delta_g)w=0\quad \mathrm{in}\; Q,
\\
w(0,\cdot)=0,
\\
\partial_tw(0,\cdot)=0,
\\
w=h \quad \mathrm{on}\;  \Sigma
\end{array}
\right.
\]
admits a unique solution $u^g(h)\in C([0,\tau],H^1(\Omega))\cap C^1([0,\tau],L^2(\Omega))$ such that $\partial_{\nu_g}w^g(h)\in L^2(\Sigma)$ (hidden regularity). Furthermore, the mapping 
\[
\Pi^g: h\in H_{0,}^1(\overline{\Sigma})\mapsto \partial_{\nu_g}w^g(h)\in L^2(\Sigma)
\]
is continuous. The operator $\Pi^g$ is usually called the hyperbolic DtN map associated with $g$.

As a special case of \cite[Theorem 1]{SU}, we have the following result.

\begin{theorem}\label{thm2.2}
There exist an integer $m\ge 2$, an open dense set in $C^m(\overline{\Omega})$ of simple metrics, denoted $\mathcal{O}_m$, and $0<\mu <1$ such that for all $\tilde{g}\in \mathcal{O}_m$ and $\tau >\mathrm{diam}_{\tilde{g}}(\Omega)$, there exist $\eta >0$ and $\mathbf{r}>0$ with the property that if $g_1,g_2\in \tilde{\mathbf{G}}$ satisfy
\begin{equation}\label{2.17}
\| g_j-\tilde{g}\|_{C(\overline{\Omega})}\le \eta,\quad \| g_j\|_{C^m(\overline{\Omega})}\le \mathbf{r},\quad j=1,2,
\end{equation}
then one can find a $C^3$-diffeomophism $F:\overline{\Omega}\rightarrow \overline{\Omega}$ satisfying $F_{|\Gamma}=I$ such that
\begin{equation}\label{2.18}
\|g_1-F^\ast g_2\|_{C^2(\overline{\Omega})}\le\tilde{c}\left\| \Pi^1-\Pi^2  \right\|_{\mathscr{B}(H_{0,}^1(\Sigma), L^2(\Sigma))}^\mu, 
\end{equation}
where the constant $\tilde{c}>0$ depends only on $n$, $\Omega$, $\tau$, $\tilde{g}$, $m$, $\mu$, $\eta$ and $\mathbf{r}$.
\end{theorem}

We point out that an earlier version of Theorem \ref{thm2.2} was established by the same authors in \cite{SU1} with $\mathbf{e}$ instead of $\tilde{g}$. A stability inequality from the knowledge of a partial hyperbolic Dirichlet to Neumann map around  $\mathbf{e}$ was obtained by Bellassoued in \cite{Bella}.

Fix $\tilde{g}\in \mathcal{O}_m$ and let $\mathcal{G}$ be the set of all $g\in \tilde{\mathbf{G}}$ satisfying \eqref{2.17}. In what follows, by shrinking if necessary $\eta$ and $\mathbf{r}$, we  assume that $\mathcal{G}\subset \mathbf{G}$, for some $\alpha >1$ and $\beta >0$ depending on $\eta$ and $\mathbf{r}$.

\begin{theorem}\label{thm2.3}
Let $g_1,g_2\in\mathcal{G}$ such that $\delta_0(g_1,g_2)<\infty$ and
\begin{equation}\label{mm1}
\lambda_k^1=\lambda_k^2,\quad \psi_k^1=\psi_k^2,\quad k\ge 1,
\end{equation}
Then there exists a $C^3$-diffeomophism $F:\overline{\Omega}\rightarrow \overline{\Omega}$ satisfying $F_{|\Gamma}=I$ such that $g_1=F^\ast g_2$.
\end{theorem}

Let $\Sigma$ be a measurable set of $\Gamma$ having the property that there exist $\tilde{x}\in \Gamma$ and $\tilde{r}>0$ such that
\[
|\Sigma \cap B(\tilde{x},r)|>0,\; \quad 0<r\le \tilde{r},
\]
where $B(\tilde{x},r)$ denotes the ball of center $\tilde{x}$ and radius $r$. Let $N$ be an arbitrary fixed neighborhood of $\Gamma$ in $\overline{\Omega}$.

\begin{corollary}\label{cor1.1}
Let $g_1,g_2\in\mathcal{G}$ such that $g_1{_{|N}}=g_2{_{|N}}$, $\delta_0(g_1,g_2)<\infty$ and
\[
\lambda_k^1=\lambda_k^2,\quad \psi_k^1{_{|\Sigma}}=\psi_k^2{_{|\Sigma}},\quad k\ge 1.
\]
Then  there exists a $C^3$-diffeomophism $F:\overline{\Omega}\rightarrow \overline{\Omega}$ satisfying $F_{|N}=I$ such that $g_1=F^\ast g_2$.
\end{corollary}

\begin{proof}
Let $g=g_1{_{|N}}=g_2{_{|N}}$. As $\lambda_k^1=\lambda_k^2$ for all $k\ge 1$, we have
\[
-\Delta_g\left(\phi_k^1-\phi_k^2\right)=\lambda_k^1\left(\phi_k^1-\phi_k^2\right)\; \mathrm{in}\; N,\quad k\ge 1.
\]
This, $\psi_k^1{_{|\Sigma}}=\psi_k^2{_{|\Sigma}}$ for all $k\ge 1$, combined with \cite[(2.20)]{CY} (uniqueness of continuation from the Cauchy data on $\Sigma$), imply that $\phi_k^1=\phi_k^2$ in $N$ for all $k\ge 1$. Therefore, \eqref{mm1} holds. We complete the proof by applying Theorem \ref{thm2.3}.
\end{proof}

\begin{remark}
{\rm
By combining  Theorem \ref{mt1} with $j=0$ and \cite[Corollary E]{LU}, we observe that Theorem \ref{thm2.3} remains valid under the following assumptions: $\Omega$ is bounded strictly convex domain with real-analytic boundary, $g_1$, $g_2$ are real-analytic and belong to some neighborhood of $\mathbf{e}$ in $C^2$. In this case the diffeomorphism $F$ is real-analytic.
}
\end{remark}

\subsection{Comments} \label{sb1.3}

In the present work, we extend and improve the initial results obtained in \cite{CY} regarding the determination of higher-order coefficients of an elliptic operator from  the corresponding BSD. We refine the analysis in \cite{CY} by placing more emphasis on the relationship between the BSD and DtN maps. We also show how to modify our analysis to establish stability inequalities for determining the conductivity or potential in an elliptic operator from  the corresponding BSD.

The uniqueness of the determination of a Riemannian manifold from the corresponding BSD goes back to Belishev and Kurylev \cite{Bel92}. Their proof is based on the so-called boundary control (BC) method.

More is known about the determination of the potential and magnetic field appearing in a Schr\"odinger equation from the corresponding BSD. The first uniqueness result was proved by Nachman, Sylvester, and Uhlamnn \cite{NSU}. The corresponding stability inequality was proved by Alessandrini and Sylvester \cite{AS} (see also \cite{Ch09} where the stability inequality is reformulated in a suitable topology). Recently, stability inequalities have been established in \cite{BCDKS} for a magnetic Schr\"odinger equation in a simple Riemannian manifold based on the Isozaki method, which was introduced in \cite{Is}. This method is inspired by the Born approximation method. Very recently, Liu, Quan, Saksala and Yan \cite{LQSY} extended the result of \cite{Ch09} to a magnetic Schrödinger equation in a simple Riemannian manifold.

The case of the Schr\"odinger equation with unbounded potential was first studied by P\"aiv\"arinta and Serov \cite{PS}, where a uniqueness result was established. The uniqueness result for the optimal class of unbounded potentials was obtained by Pohjola \cite{Po}. The corresponding stability inequality was proved by the author in \cite{Ch24}. The case of the Schr\"odinger equation with unbounded potential and the Robin boundary condition were recently studied in \cite{CMS1,CMS2}, where several uniqueness and stability results were obtained.

Since the pioneering article by Nachman, Sylvester and Uhlamnn \cite{NSU}, the literature on the multidimensional inverse spectral problems has continued to grow. Without being exhaustive, we refer to \cite{BCY,BCY2,BD,BKMS,Ch1,ChS,KK,KKL,KKS,Ki1,KOM,KS,So}.

It should be noted that the method used in this article is indirect. A direct method, such as that introduced by Isozaki \cite{Is} in the isotropic case, appears difficult to implement in the anisotropic case to obtain stability inequalities for the determination of higher-order coefficients of an elliptic operator from the corresponding BSD. 

%The stability corresponding to the uniqueness in Theorem \ref{thm2.3} still an open problem as well as the case of partial BSD, which means that a finite number of eigenvalues and trace of normal derivative of eigenfunction are supposed unknown.

The issue of stability, in relation to the uniqueness result of Theorem \ref{thm2.3}, remains an open problem, as does the case of the partial BSD, where a finite number of eigenvalues and the trace of the normal derivative of the eigenfunctions are assumed to be unknown.

\section{Preliminaries}\label{s2}

\subsection{Weyl's asymptotic formula}\label{sb2.1}

\begin{lemma}\label{lem1}
There exists a constant $\varkappa >1$  depending only on $\Omega$ such that 
\begin{equation}\label{ee}
\vartheta ^{-1} k^{\frac{2}{n}}\le \lambda_k^g \le \vartheta k^{\frac{2}{n}},\quad k\ge 1,\; g\in \mathbf{G},
\end{equation}
where $\vartheta =\varkappa\alpha^{\frac{n+2}{2}}$.
\end{lemma}

\begin{proof}
Denote by $(\mu_k)$ the sequence of eigenvalues of $A^\mathbf{e}$. Let $V_k$ be the set of all subspaces of $H_0^1(\Omega)$ of dimension $k$. According to the min-max principle, we have
\begin{equation}\label{ee1}
\alpha^{-\frac{n+2}{2}} \mu_k\le \min_{E\in V_k}\max_{u\in E\setminus\{0\}}\frac{\||g|^{\frac{1}{4}}\sqrt{g^{-1}}\nabla u\|_{0,2}^2}{\||g|^{\frac{1}{4}}u\|_{0,0}^2}=\lambda_k^g  \le \alpha^{\frac{n+2}{2}}\mu_k.
\end{equation}
On the other hand, we know that there exists a constant $\varkappa >1$ depending only on $\Omega$ such that
\[
\varkappa^{-1}k^{\frac{2}{n}}\le \mu_k\le \varkappa k^{\frac{2}{n}}.
\]
This and \eqref{ee1} yield \eqref{ee}.
\end{proof}

\subsection{Useful formula}\label{sb2.2}

Let $g\in \mathbf{G}$.  Recall that the resolvent set of $-A^g$, usually denoted $\rho(-A^g)$, is given by
\[
\rho(-A^g)=\mathbb{C}\setminus\{-\lambda_k^g;\; k\ge 1\}.
\]
For later use, note that $(-\vartheta^{-1},\infty )\subset \rho(-A^g)$ for all $g\in \mathbf{G}$, where $\vartheta$ is as in Lemma \ref{lem1}.

For $g\in \mathbf{G}$, $\varphi \in H^{\frac{3}{2}}(\Gamma)$ and $\lambda \in  \rho(-A^g)$, $u ^g(\lambda)(\varphi)\in H^2(\Omega)$ will denote the solution of the BVP
\[
-\Delta_gu+\lambda u=0\; \mathrm{in}\; \Omega ,\quad u_{|\Gamma}=\varphi.
\]

\begin{lemma}\label{lem2}
For all $\lambda \in  \rho(-A^g)$, $g\in \mathbf{G}$ and $\varphi \in H^{\frac{3}{2}}(\Gamma)$, we have
\begin{equation}\label{f1}
u^g(\lambda)(\varphi)=-\sum_{k\ge 1}\frac{\langle\varphi| \psi_k^g\rangle_g}{\lambda_k^g+\lambda}\phi_k^g.
\end{equation}
\end{lemma}

\begin{proof} 
Let $\lambda\in \rho(-A^g)$, $g\in \mathbf{G}$, $\varphi \in H^{\frac{3}{2}}(\Gamma)$ and $u:=u^g(\lambda)(\varphi)$. Applying twice Green's formula, we get
\begin{align*}
-\lambda \int_\Omega u\overline{\phi_k^g} dV_g=-\int_\Omega \Delta_g u\overline{\phi_k^g} dV_g&  =\int_\Omega \langle \nabla_g u,\nabla_g \overline{\phi_k^g}\rangle_g dV_g
\\
&=-\int_\Omega u \Delta_g \overline{\phi_k^g}  dV_g+\int_\Gamma \varphi \partial_{\nu_g} \overline{\phi_k^g} dS_g
\\
&=\lambda_k^g\int_\Omega u \overline{\phi_k^g}  dV_g+\int_\Gamma \varphi \partial_{\nu_g} \overline{\phi_k^g} dS_g .
\end{align*}
That is, we have
\[
(u|\phi_k^g)_g=-\frac{\langle \varphi|\psi_k^g\rangle_g}{\lambda_k^g+\lambda},\quad k\ge 1.
\]
This identity, combined with  
\[
u=\sum_{k\ge 1}(u|\phi_k^g)_g\phi_k^g,
\]
gives the expected formula.
\end{proof}

\subsection{Elliptic a priori estimates}\label{sb2.3}

In the following, the usual norm of $H ^2(\Omega)$ will be denoted $\|\cdot\|_{2,2}$.

\begin{lemma}\label{lemEae}
We have for all $g\in \mathbf{G}$ and $f\in L^2(\Omega)$ 
\begin{equation}\label{el3}
\|(A^g)^{-1}f\|_{2,2}\le \mathbf{c} \mathbb{c}\|f\|_{0,2}.
\end{equation}
\end{lemma}

\begin{proof}
Let $g\in \mathbf{G}$, $f\in L^2(\Omega)$ and $u=(A^g)^{-1}f\in H_0^1(\Omega)\cap H^2(\Omega)$. Then we have
\[
g^{k\ell}\partial_{k\ell}^2u+ b^j \partial_j u=-f,
\]
where 
\[
b^j =\frac{1}{\sqrt{|g|}}\partial_k(\sqrt{|g|}g^{kj}),\quad 1\le j \le n.
\]
We verify that 
\[
\sup_j \|b^j\|_{C^0(\overline{\Omega})}\le \tilde{\beta},
\]
where $\tilde{\beta}>0$ is a constant depending only on $\alpha$ and $\beta$.

It follows from \cite[Theorem 9.14]{GT} that there exists a constant $\mu>0$  depending only on $n$, $\Omega$, $\alpha$ and $\beta$ such that
\begin{equation}\label{el1}
\|u\|_{2,2}\le \mathbf{c}\|f -\mu u\|_{0,0}.
\end{equation}
On the other hand, making an integration by part, we obtain
\[
\int_{\Omega}\sqrt{|g|}g^{-1}\nabla u\cdot \nabla \overline{u}dx=\int_\Omega \sqrt{|g|}f\overline{u} dx.
\]
Hence
\[
\alpha^{-\frac{n+2}{2}} \int_{\Omega}|\nabla u|^2dx\le \int_\Omega \sqrt{|g|} f\overline{u} dx\le \alpha^{\frac{n}{2}}\int_\Omega |fu| dx.
\]
This and Poincaré's inequality  imply
\begin{equation}\label{el2}
\|u\|_{0,0}\le \mathbf{c}\|f\|_{0,0}.
\end{equation}
Putting together \eqref{el1} and \eqref{el2} yield the expected inequality.
\end{proof}

Next, let $\lambda \in \mathbb{R}$, $g\in \mathbf{G}$ and $u\in H_0^1(\Omega)\cap H^2(\Omega)$ satisfying $\|u\|_{0,2,g}=1$ and $-\Delta_gu=\lambda u$ in $\Omega$. As $u=\lambda (A^g)^{-1} u$, applying \eqref{el3}, we obtain
\begin{equation}\label{el4}
\|u\|_{0,2}\le \mathbf{c}|\lambda| .
\end{equation}
Proceeding as in the proof of \eqref{el2}, we get
\begin{equation}\label{el5}
\|u\|_{1,2}\le \mathbf{c} |\lambda|^{\frac{1}{2}}.
\end{equation}

In accordance with the the previous notations,  the norm of $H^{\frac{7}{4}}(\Omega)$  will be denoted $\|\cdot \|_{\frac{7}{4},2}$. From \cite[Theorem 9.6]{LM1}, we have the following interpolation inequality.
\[
\|w\|_{\frac{7}{4},2}\le c_0\|w\|_{1,2}^{\frac{1}{4}}\|w\|_{2,2}^{\frac{3}{4}},\quad w\in H^2(\Omega),
\]
where the constant $c_0>0$ depends only on $n$ and $\Omega$. 

A combination of \eqref{el4} and \eqref{el5} gives
\begin{equation}\label{el6}
\|u\|_{\frac{7}{4},2}\le \mathbf{c}|\lambda|^{\frac{7}{8}},
\end{equation}

In particular, we have from \eqref{el6}
\[
\|\phi_k^g\|_{\frac{7}{4},2}\le \mathbf{c}(\lambda_k^g)^{\frac{7}{8}},\quad k\ge 1,\;  g\in \mathbf{G},
\]
and as the trace operator $w\in H^{\frac{7}{4}}(\Omega)\mapsto \partial_{\nu_g}w\in L^2(\Gamma)$ is continuous, we have also
\begin{equation}\label{el8}
\|\psi_k^g\|_{0,2}\le \mathbf{c}(\lambda_k^g)^{\frac{7}{8}}\quad k\ge 1,\;  g\in \mathbf{G}.
\end{equation}

In light of \eqref{ee}, we have from \eqref{el8} the following result.

\begin{lemma}\label{trb1}
The following estimate holds.
\begin{equation}\label{el10}
\|\psi_k^g\|_{0,2}\le \mathbf{c}k^{\frac{7}{4n}}\quad k\ge 1,\; g\in \mathbf{G}.
\end{equation}
\end{lemma}

\begin{remark}
{\rm
Note that, by simply adjusting some exponents, the previous proof remains valid if we replace $H^{\frac{7}{4}}(\Omega)$ with $H^{\frac{3}{2}+\epsilon}(\Omega)$, for any  $0 < \epsilon < \frac{1}{2}$. This follows from that fact that the mapping $w \in H^{\frac{3}{2}+\epsilon}(\Omega) \mapsto \partial_{\nu_g}w \in L^2(\Gamma)$ remains continuous. Consequently, all the subsequent results can be adapted by replacing $H^{\frac{7}{4}}(\Omega)$ with $H^{\frac{3}{2}+\epsilon}(\Omega)$, where $0<\epsilon <\frac{1}{2}$ is arbitrarily chosen.
}
\end{remark}

\subsection{Analyticity properties}\label{sb2.4}

We combine the usual definition of fractional powers of $-A^g$, $g\in \mathbf{G}$,  with \cite[Theorem 1]{Fu} and its proof to verify that $H^{\frac{7}{4}}(\Omega)\cap H_0^1(\Omega)$ can be identified, algebraically and topologically, with 
\[
\left\{ w\in L^2(\Omega);\sum_{k\ge 1} (\lambda_k ^g)^{\frac{7}{4}}|(w|\phi_k^g)_g|^2<\infty\right\}.
\]
Note that, according to \eqref{ee},  $H^{\frac{7}{4}}(\Omega)\cap H_0^1(\Omega)$ can also be identified, algebraically and topologically, with
\[
\left\{ w\in L^2(\Omega);\sum_{k\ge 1} k^{\frac{7}{2n}}|(w|\phi_k^g)_g|^2<\infty\right\},
\]
and  the following two norms are equivalent on $H^{\frac{7}{4}}(\Omega)\cap H_0^1(\Omega)$ 
\[
w\mapsto \left(\sum_{k\ge 1} (\lambda_k ^g)^{\frac{7}{4}}|(w|\phi_k^g)_g|^2\right)^{\frac{1}{2}},
\quad w\mapsto \left(\sum_{k\ge 1} k^{\frac{7}{2n}}|(w|\phi_k^g)_g|^2\right)^{\frac{1}{2}}.
\]

These preliminary observations will be used hereinafter.

For $\lambda \in \rho(-A^g)$, $\Lambda^g(\lambda)\in \mathscr{B}(H^{\frac{3}{2}}(\Gamma),L^2(\Gamma))$ will denote the DtN map given by
\[
\Lambda^g(\lambda):\varphi\in H^{\frac{3}{2}}(\Gamma)\mapsto \partial_{\nu_g}u^g(\lambda)\in L^2(\Gamma).
\]

The following notations will be used later.
\[
R^g(\lambda ):=(A^g+\lambda)^{-1},\quad \lambda \in \rho(-A^g),
\]
and $X^{(j)}:=\frac{d^jX}{d\lambda^j}$, $j\ge 0$.  

\begin{proposition}\label{prohol}
For all $\varphi \in H^{\frac{3}{2}}(\Gamma)$, the mappings 
\begin{align*}
&\lambda\in \rho(-A^g)\mapsto u^g(\lambda)(\varphi)\in H^2(\Omega),
\\
&\lambda\in \rho(-A^g)\mapsto \Lambda^g(\lambda)(\varphi)\in L^2(\Gamma)
\end{align*}
are analytic. Furthermore, for all $\lambda \in \rho(-A^g)$, we have
\begin{align}
&(u^g)^{(j)}(\lambda)(\varphi)=(-1)^{j+1}j!\sum_{k\ge 1}\frac{\langle \varphi|\psi^g_k\rangle_g}{(\lambda_k^g+\lambda)^{j+1}}\phi_k^g,\quad j\ge 0,\label{ser1}
\\
&(\Lambda^g)^{(j)}(\lambda)(\varphi)=(-1)^{j+1}j!\sum_{k\ge 1}\frac{\langle \varphi|\psi^g_k\rangle_g}{(\lambda_k^g+\lambda)^{j+1}}\psi_k^g,\quad j>\frac{n+3}{4},\label{ser3}
\end{align}
where the series in the right hand side of \eqref{ser3} converges in $L^2(\Gamma)$ uniformly with respect to $\lambda$ in a bounded set of $\rho(-A^g)$.
\end{proposition}

\begin{proof}
Let $\varphi \in H^{\frac{3}{2}}(\Gamma)$ and $f\in L^2(\Omega)$. Since $\lambda \in \rho(-A^g)\mapsto (-\Delta_g+\lambda)E\varphi \in L^2(\Omega)$ is an affine function and $\lambda\in \rho(-A^g)\mapsto R^g(\lambda)f\in H^2(\Omega)$ is analytic (the proof is similar to that of \cite[Proposition 2.30 (ii)]{Ch09}), we derive that 
\[
\lambda\in \rho(-A^g) \mapsto u^g(\lambda)(\varphi)=-R^g(\lambda)((-\Delta_g+\lambda)E\varphi)+E\varphi \in H^2(\Omega)
\]
is analytic. We verify that $(u^g)^{(j)}(\lambda)(\varphi)$, $j\ge 1$, is the solution of the BVP
\begin{equation}\label{spe}
(-\Delta_g +\lambda )v=-j(u^g)^{(j-1)}(\lambda)(\varphi)\; \mathrm{in}\; \Omega,\quad v_{|\Gamma}=0.
\end{equation}
That is we have for $1\le k\le j-1$
\begin{align*}
(u^g)^{(j)}(\lambda)(\varphi)&=-jR^g(\lambda)((u^g)^{(j-1)}(\lambda)(\varphi))
\\
& =(-1)^2j(j-1)(R^g)^2(\lambda)((u^g)^{(j-2)}(\lambda)(\varphi))
\\
&\ldots
\\
& =(-1)^kj(j-1)\ldots (j-k)(R^g)^k(\lambda)((u^g)^{(j-k+1)}(\lambda)(\varphi)).
\end{align*}
Therefore, we obtain the formula
\begin{equation}\label{F1}
(u^g)^{(j)}(\lambda)(\varphi)=(-1)^jj!(R^g)^j(\lambda)(u^g(\lambda)(\varphi)),\quad j\ge 1.
\end{equation}
We have
\[
(R^g)^j(\lambda)(f)=\sum_{k\ge 1}\frac{(f|\phi_k^g)_g}{(\lambda_k^g+\lambda)^j}\phi_k^g
\]
and it follows \eqref{f1} that
\[
u^g(\lambda)(\varphi)=-\sum_{k\ge 1}\frac{\langle \varphi|\psi^g_k\rangle_g}{\lambda_k^g+\lambda}\phi_k^g.
\]
Whence
\[
(u^g)^{(j)}(\lambda)(\varphi)=(-1)^{j+1}j!\sum_{k\ge 1}\frac{\langle \varphi|\psi^g_k\rangle_g}{(\lambda_k^g+\lambda)^{j+1}}\phi_k^g,\quad j\ge 0.
\]
That is we proved \eqref{ser1}.

On the other hand, the following inequality
\[
\left|\frac{\langle \varphi|\psi^g_k\rangle_g}{(\lambda_k^g+\lambda)^{j+1}}\right|\le \frac{1}{|\lambda_k^g+\lambda|^{j+1}}\|\psi_k^g\|_{0,0,g}\|\varphi\|_{0,0,g}.
\]
and \eqref{el10} yield
\[
k^{\frac{7}{2n}}\left|\frac{\langle \varphi|\psi^g_k\rangle_g}{(\lambda_k^g+\lambda)^{j+1}}\right|^2 =O\left(k^{-\frac{4j-3}{n}}\right),
\]
uniformly with respect to $\lambda$ in a bounded set of $\rho(-A^g)$. In consequence, the series in the right hand side of \eqref{ser1} converges in $H^{\frac{7}{4}}(\Omega)$ uniformly with respect to $\lambda$ in a bounded set of $\rho(-A^g)$ provided that $j>\frac{n+3}{4}$. Using that the mapping $w\in H^{\frac{7}{4}}(\Omega)\mapsto \partial_{\nu_g}w\in L^2(\Gamma)$ is continuous, we deduce that for all $j>\frac{n+3}{4}$
\begin{equation}\label{ser2}
(\partial_{\nu_g}u^g)^{(j)}(\lambda)(\varphi)=\partial_{\nu_g}(u^g)^{(j)}(\lambda)(\varphi)=(-1)^{j+1}j!\sum_{k\ge 1}\frac{\langle \varphi|\psi^g_k\rangle_g}{(\lambda_k^g+\lambda)^{j+1}}\psi_k^g,
\end{equation}
and the series in the right hand side of \eqref{ser2} converges in $L^2(\Gamma)$ uniformly with respect to $\lambda$ in a bounded set of $\rho(-A^g)$. In other words, we proved that 
\begin{equation}\label{ser2.0}
(\Lambda^g)^{(j)}(\lambda)(\varphi)=(-1)^{j+1}j!\sum_{k\ge 1}\frac{\langle \varphi|\psi^g_k\rangle_g}{(\lambda_k^g+\lambda)^{j+1}}\psi_k^g, \quad j>\frac{n+3}{4},
\end{equation}
where the series in the right-hand side of \eqref{ser2.0} converges in $L^2(\Gamma)$ uniformly with respect to $\lambda$ in a bounded set of $\rho(-A^g)$. The proof is complete
\end{proof}

\section{BSD determines a family of elliptic DtN maps}\label{s3}

 In what follows, if $j\ge 0$ is an integer, then $\mathbf{c}_j>0$ will denote a generic constant depending only on $n$, $\Omega$, $\alpha$, $\beta$ and $j$.
 
In this section, we aim to prove the following result.

\begin{theorem}\label{mt1}
Let $g_1,g_2\in \mathbf{G}_0$ be such that $\delta_+(g_1,g_2)<\infty$ and $j\ge 0$ be an integer. Then we have for all $\varphi \in H^{\frac{3}{2}}(\Gamma)$
\begin{equation}\label{mte1}
\|(\Lambda^1)^{(j)}(0)(\varphi)-(\Lambda^2)^{(j)}(0)(\varphi)\|\le \mathbf{c}_j \delta (g_1,g_2)\|\varphi\|_{0,2}.
\end{equation}
\end{theorem}

Before proving this theorem, we prove the following technical lemma.

\begin{lemma}\label{lemte}
Let $j\ge 2$ and $m\ge 0$ be two integers, $0<a<b$ and $\lambda >0$. Then we have
\begin{align}
&\int_0^\lambda \frac{s^{j-1}}{(a+s)^{j+m+1}}ds\le \frac{j2^{j-1}}{a^{m+1}}, \label{te1}
\\
&\int_0^\lambda \left[\frac{s^{j-1}}{(a+s)^{j+1}}-\frac{s^{j-1}}{(b+s)^{j+m+1}}\right]ds\le \frac{j(j+m+1)2^{j-1}(b-a)}{a^{m+1}b}.\label{te2}
\end{align}
\end{lemma}

\begin{proof}
Since 
\[
\int_0^\lambda \frac{s^{j-1}}{(a+s)^{j+m+1}}ds\le \frac{1}{a^m} \int_0^\lambda \frac{s^{j-1}}{(a+s)^{j+1}}ds,
\]
it is enough to prove \eqref{te1} when $m=0$. Making the change of variable $\tau=1+s/a$, we obtain
\begin{equation}\label{te3}
\int_0^\lambda \frac{s^{j-1}}{(a+s)^{j+1}}ds=\frac{1}{a}\int_1^{1+\lambda/a}\frac{(\tau-1)^{j-1}}{\tau^{j+1}}d\tau.
\end{equation}
We compute the integral in the right hand side of \eqref{te3}. We have, where $C_k^\ell=\frac{k!}{\ell!(k-\ell)!}$,
\begin{align*}
\int_1^{1+\lambda/a}\frac{(\tau-1)^{j-1}}{\tau^{j+1}}d\tau&=\sum_{\ell=0}^{j-1}C_{j-1}^\ell (-1)^{j-1-\ell}\int_1^{1+\lambda/a}\frac{1}{\tau^{j+1-\ell}}d\tau
\\
&=\sum_{\ell=0}^{j-1}C_{j-1}^\ell (-1)^{j-1-\ell}(j-\ell)\left( 1-\frac{1}{(1+\lambda/a)^{j-\ell}} \right)
\\
&\le j2^{j-1}.
\end{align*}
This in \eqref{te3} gives \eqref{te1}.

To prove \eqref{te2}, we use
\begin{align*}
&\frac{1}{(a+\lambda)^{j+m+1}}-\frac{1}{(b+\lambda)^{j+m+1}} 
\\
&\hskip 1.5cm= \frac{b-a}{(a+\lambda)(b+\lambda)} \left( \frac{1}{(a+\lambda)^{j+m}}+\frac{1}{(a+\lambda)^{j+m-1}}\frac{1}{(b+\lambda)}\right.
\\
&\hskip 6cm \left.  +\ldots +\frac{1}{a+\lambda}\frac{1}{(b+\lambda)^{j+m-1}}+\frac{1}{(b+\lambda)^{j+m}}\right)
\\
&\hskip 1.5cm \le \frac{b-a}{b}\frac{j+m+1}{(a+\lambda)^{j+m+1}}
\end{align*}
to obtain
\[
\int_0^\lambda \left[\frac{s^{j-1}}{(a+s)^{j+1}}-\frac{s^{j-1}}{(b+s)^{j+1}}\right]ds\le \frac{(j+m+1)(b-a)}{b}\int_0^\lambda \frac{s^{j-1}}{(a+s)^{j+m+1}}ds.
\]
This and \eqref{te1} imply \eqref{te2}.
\end{proof}

The following lemma will also be used in the proof of Theorem \ref{mt1}.

\begin{lemma}\label{lemo1}
Let $g_1,g_2\in \mathbf{G}_0$ such that $\delta_+(g_1,g_2)<\infty$ and $\varphi\in H^{\frac{3}{2}}(\Gamma)$. Then
\begin{equation}\label{o0}
\lim_{\lambda \rightarrow \infty}\left\|\lambda^{j}\left[(\Lambda^1)^{(j)}(\lambda)(\varphi)-(\Lambda^2)^{(j)}(\lambda)(\varphi)\right]\right\|_{0,2}=0,\quad j\ge 0.
\end{equation}
\end{lemma}

\begin{proof}
For $\lambda >0$, let $u(\lambda):=u^1(\lambda)(\varphi)-u^2(\lambda)(\varphi)$. We split $u^{(j)}(\lambda)$ into three terms 
\[
u^{(j)}(\lambda)=\sum_{m=1}^3u_{m,j}(\lambda),
\]
where
\begin{align}
&u_{1,j}(\lambda)(\varphi):=(-1)^{j+1}j!\sum_{k\ge 1} \left(\frac{1}{(\lambda +\lambda_k^1)^{j+1}}-\frac{1}{(\lambda +\lambda_k^2)^{j+1}}\right)\langle\varphi|\psi_k^1\rangle_{g_0}\phi_k^1,\label{o4}
\\
&u_{2,j}(\lambda)(\varphi):=(-1)^{j+1}j!\sum_{k\ge 1} \frac{1}{(\lambda +\lambda_k^2)^{j+1}}\langle\varphi|\psi_k^1-\psi_k^2\rangle_{g_0}\phi_k^1,\label{o5}
\\
&u_{3,j}(\lambda)(\varphi):=(-1)^{j+1}j!\sum_{k\ge 1} \frac{1}{(\lambda +\lambda_k^2)^{j+1}}\langle\varphi|\psi_k^2\rangle_{g_0}(\phi_k^1-\phi_k^2).\label{o6}
\end{align}

Since
\begin{align}
&j!\left(\frac{1}{(\lambda +\lambda_k^1)^{j+1}}-\frac{1}{(\lambda +\lambda_k^2)^{j+1}}\right)\label{ui0}
\\
&\hskip 2cm =-(j+1)!\int_0^1\frac{\lambda_k^1-\lambda_k^2}{(\lambda +\lambda_k^2+t(\lambda_k^1-\lambda_k^2))^{j+2}}dt, \nonumber
\end{align}
we obtain
\begin{equation}\label{ui1}
j!\left|\frac{1}{(\lambda +\lambda_k^1)^{j+1}}-\frac{1}{(\lambda +\lambda_k^2)^{j+1}}\right|\le \mathbf{c}(j+1)!\lambda^{-j}k^{-\frac{4}{n}}|\lambda_k^1-\lambda_k^2|.
\end{equation}
If
\[
\rho_{j,k}(\lambda):=(-1)^{j+1}j!\left(\frac{1}{(\lambda +\lambda_k^1)^{j+1}}-\frac{1}{(\lambda +\lambda_k^2)^{j+1}}\right)\langle\varphi|\psi_k^1\rangle_{g_0},
\]
then we get
\[
k^{\frac{7}{4n}}\lambda ^j|\rho_{j,k}|\le \mathbf{c}(j+1)!k^{-\frac{1}{2n}}|\lambda_1-\lambda_2|\|\varphi\|_{0,2}.
\]
Let $2<\tilde{p}<\frac{2n}{n-1}$. Applying H\"older's inequality, we find
\[
\sum_{k\ge \ell}\left[ k^{\frac{7}{4n}}\lambda ^j|\rho_{j,k}(\lambda)|\right]^2\le \mathbf{c}[(j+1)!]^2\left(\sum_{k\ge1}k^{-\frac{\tilde{p}}{n(\tilde{p}-2)}}\right)^{\frac{p-2}{p}}\|(\lambda_k^1-\lambda_k^2)\|_{\ell^{\tilde{p}}}^2\|\varphi\|_{0,2}^2
\]
and hence
%As $(k^{-\frac{1}{2n}}(\lambda_1-\lambda_2))\in \ell^1\subset \ell^2$ and $\|\cdot\|_{\ell^2}\le \|\cdot \|_{\ell ^1}$, we find that $(k^{\frac{7}{4n}}\lambda ^j \rho_{j,k})\in \ell ^2$ and 
\[
\sum_{k\ge \ell}\left[ k^{\frac{7}{4n}}\lambda ^j|\rho_{j,k}(\lambda)|\right]^2\le \mathbf{c}[(j+1)!\delta(g_1,g_2)\|\varphi\|_{0,2}]^2,
\]
where we used that $p<\tilde{p}$. Therefore, the series in the right hand side of \eqref{o4} converges in $H^{\frac{7}{4}}(\Omega)$, and since $k^{\frac{7}{4n}}\lambda ^j|\rho_{j,k}(\lambda)|$ tends to $0$ when $\lambda$ goes to $\infty$, we have, according to the dominated convergence theorem,
\[
\lim_{\lambda \rightarrow \infty}\|\lambda^ju_{1,j}(\lambda)(\varphi)\|_{\frac{7}{4},2}=0.
\]
In particular, as the trace map $w\in H^{\frac{7}{4}}(\Omega)\mapsto \partial_{\nu_{g_0}}w\in L^2(\Gamma)$ is continuous, we obtain
\[
\|\lambda^j\partial_{\nu_{g_0}} u_{1,j}(\lambda)(\varphi)\|_{0,2}\le \mathbf{c}(j+1)!\delta(g_1,g_2)\|\varphi\|_{0,2}
\]
and
\begin{equation}\label{o7}
\lim_{\lambda \rightarrow \infty}\|\lambda^j\partial_{\nu_{g_0}}u_{1,j}(\lambda)(\varphi)\|_{0,2}=0.
\end{equation}

For the series in \eqref{o5}, we have 
\[
\left|j!\frac{k^{\frac{7}{4n}}\lambda^j}{(\lambda +\lambda_k^2)^{j+1}}\langle\varphi|\psi_k^1-\psi_k^2\rangle_{g_0}\right|^2
\le \mathbf{c}[j!]^2k^{-\frac{1}{2n}}\|\psi_k^1-\psi_k^2\|_{0,2}^2\|\varphi\|_{0,2}^2.
\]
Let $2<\tilde{q}<\frac{4n}{2n-1}$. Applying again H\"older's inequality, we obtain
\begin{align*}
&\sum_{k\ge 1}\left|j!\frac{k^{\frac{7}{4n}}\lambda^j}{(\lambda +\lambda_k^2)^{j+1}}\langle\varphi|\psi_k^1-\psi_k^2\rangle_{g_0}\right|^2
\\
&\hskip 2cm\le \mathbf{c}[j!]^2\left(\sum_{k\ge 1}k^{-\frac{\tilde{q}}{2n(q-2)}}\right)^{\frac{\tilde{q}-2}{\tilde{q}}}\|\psi_k^1-\psi_k^2\|_{\ell^{\tilde{q}}(L^2(\Gamma))}^2\|\varphi\|_{0,2}^2.
\end{align*}

This inequality at hand and using that $q<\tilde{q}$, we proceed as for \eqref{o7} to derive
\[
\|\lambda^j\partial_{\nu_{g_0}} u_{2,j}(\lambda)(\varphi)\|_{0,2}\le \mathbf{c}j!\delta(g_1,g_2)\|\varphi\|_{0,2}
\]
and
\begin{equation}\label{o8}
\lim_{\lambda \rightarrow \infty}\|\lambda^j\partial_{\nu_{g_0}}u_{2,j}(\lambda)(\varphi)\|_{0,2}=0.
\end{equation}

Next, we have for all $\ell\ge 1$
\[
(u_{3,j}(\lambda)(\varphi)|\phi_\ell^1)_{g_1}=(-1)^{j+1}j!\sum_{k\ge 1} \frac{1}{(\lambda +\lambda_k^2)^{j+1}}\langle\varphi|\psi_k^2\rangle_{g_0}(\delta_{k\ell}-(\phi_k^2|\phi_\ell^1)_{g_1}).
\]
Whence
\[
\ell^{\frac{7}{4n}}\lambda^j|(u_{3,j}(\lambda)(\varphi)|\phi_\ell^1)_{g_1}|\le \mathbf{c}j!\sum_{k\ge 1}\ell^{\frac{7}{4n}}k^{-\frac{1}{4n}}|(\delta_{k\ell}-(\phi_k^2|\phi_\ell^1)_{g_1})|\|\varphi\|_{0,2}.
\]
As $(\ell^{\frac{7}{4n}}k^{-\frac{1}{4n}}(\delta_{k\ell}-(\phi_k^2|\phi_\ell^1)_{g_1}))\in \ell^1(\mathbb{N}^2)$, we get $(\ell^{\frac{7}{4n}}\lambda^j(u_{3,j}(\lambda)(\varphi)|\phi_\ell^1)_{g_1})\in \ell^1\subset \ell^2$ (with $\|\cdot \|_{\ell^2}\le \|\cdot \|_{\ell^1}$). Again, we proceed as for \eqref{o7} to obtain
\begin{equation}\label{o9}
\lim_{\lambda \rightarrow \infty}\|\lambda^j\partial_{\nu_{g_0}}u_{3,j}(\lambda)(\varphi)\|_{0,2}=0.
\end{equation}
Putting together \eqref{o7}, \eqref{o8} and \eqref{o9}, we obtain \eqref{o0}.
\end{proof}

\begin{proof}[Proof of Theorem \ref{mt1}]
Let $j_0=\lfloor\frac{n+3}{4}\rfloor +1$ and $\varphi \in H^{\frac{3}{2}}(\Gamma)$  such that $\|\varphi\|_{0,2}=1$. 
Let $m\ge 0$ be an integer and $\lambda >0$ and set 
\[
\Lambda(\lambda):= \Lambda^1(\lambda)(\varphi)-\Lambda^2(\lambda)(\varphi).  
\]
Applying Taylor's formula to the mapping $t\in [0,1] \mapsto \Lambda^{(m)} ((1-t)\lambda)$, we obtain
\begin{align}
&\Lambda^{(m)}(0)=\sum_{j=0}^{j_0-1}(-\lambda)^j\Lambda^{(m+j)} (\lambda)\label{o10}
\\
&\hskip 3.5cm+\int_0^1\frac{(1-t)^{j_0-1}(-\lambda)^{j_0}}{(j_0-1)!}\Lambda^{(m+j_0)} ((1-t)\lambda)dt.\nonumber
\end{align}
For convenience, we use hereinafter the notation
\[
\Upsilon:=\int_0^1\frac{(1-t)^{j_0-1}(-\lambda)^{j_0}}{(j_0-1)!}\Lambda^{(m+j_0)} ((1-t)\lambda)dt.
\]
From the formula \eqref{ser3}, we have
\[
\Upsilon =\sum_{k\ge 1}\left(d_k^1\langle\varphi|\psi_k^1\rangle_{g_0}\psi_k^1 -d_k^2\langle\varphi|\psi_k^2\rangle_{g_0}\psi_k^2\right),
\]
where
\[
d_k^\ell=\frac{(-1)^{m+1}(j_0+m)!}{(j_0-1)!}\int_0^1\frac{(1-t)^{j_0-1}\lambda^{j_0}}{(\lambda_k^\ell +(1-t)\lambda)^{j_0+m+1}}dt,\quad \ell=1,2.
\]
The change of variable $s=(1-t)\lambda $ in the integral above yields
\[
d_k^\ell=\frac{(-1)^{m+1}(j_0+m)!}{(j_0-1)!}\int_0^\lambda \frac{s^{j_0-1}}{(\lambda_k^\ell +s)^{j_0+m+1}}ds,\quad  k\ge 1,\;\ell=1,2.
\]
Using \eqref{te1}  and \eqref{ee}, we get
\begin{equation}\label{r1}
|d_k^\ell|\le \mathbf{c}_mk^{-\frac{2}{n}},\quad  k\ge 1,\;\ell=1,2.
\end{equation}
Similarly, a combination of \eqref{te2} and \eqref{ee} gives
\[
|d_k^1-d_k^2|\le \mathbf{c}_mk^{-\frac{4}{n}}|\lambda_k^1-\lambda_k^2|,\quad  k\ge 1.
\]
This and \eqref{el10} imply
\begin{equation}\label{r3}
\|(d_k^1-d_k^2)\langle\varphi|\psi_k^1\rangle_{g_0}\psi_k^1\|_{0,2}\le \mathbf{c}_mk^{-\frac{1}{2n}}|\lambda_k^1-\lambda_k^2|,\quad  k\ge 1.
\end{equation}
As in the preceding proof, H\"older's inequality then yields
\begin{equation}\label{r2}
\sum_{k\ge 1}\|(d_k^1-d_k^2)\langle\varphi|\psi_k^1\rangle_{g_0}\psi_k^1\|_{0,2}\le \mathbf{c}_m\|(\lambda_k^1-\lambda^2)\|_{\ell^p},\quad  k\ge 1.
\end{equation}
From \eqref{r1} and \eqref{el10}, we obtain 
\begin{align*}
&\|d_k^2\langle\varphi|\psi_k^1-\psi_k^2\rangle_{g_0}\psi_k^1\| +\|d_k^2\langle\varphi|\psi_k^2\rangle_{g_0}(\psi_k^1-\psi_k^2)\|_{0,2}
\\
&\hskip 5cm\le \mathbf{c}_mk^{-\frac{1}{4n}}\|\psi_k^1-\psi_k^2\|_{0,2},\quad  k\ge 1.
\end{align*}
Again, H\"older's inequality implies
\begin{align}
&\sum_{k\ge 1}\|d_k^2\langle\varphi|\psi_k^1-\psi_k^2\rangle_{g_0}\psi_k^1\| +\|d_k^2\langle\varphi|\psi_k^2\rangle_{g_0}(\psi_k^1-\psi_k^2)\|_{0,2}\label{r4}
\\
&\hskip 5cm\le \mathbf{c}_m\|\psi_k^1-\psi_k^2\|_{\ell^q(L^2(\Gamma)},\quad  k\ge 1.\nonumber
\end{align}
Putting together \eqref{r2} and \eqref{r4}, we get
\begin{equation}\label{r3}
\|\Upsilon\|_{0,2}\le \mathbf{c}_m \delta(g_1,g_2).
\end{equation}
On the other hand, the inequality
\[
\|\lambda^j\Lambda^{(m+j)}(\varphi)\|_{0,2}\le \|\lambda^{j+m}\Lambda^{(m+j)}(\varphi)\|_{0,2},\quad \lambda \ge 1,
\]
and \eqref{o0} yield
\begin{equation}\label{r5}
\lim_{\lambda \rightarrow \infty}\|\lambda^j\Lambda^{(m+j)}(\varphi)\|_{0,2}=0.
\end{equation} 
The expected inequality follows then from \eqref{r3} and \eqref{r5}.
\end{proof}

\begin{proof}[Completion of the proof of Theorem \ref{thm2.1}]
According to \cite[Theorem 2.1]{CS}, there exist three constants $\kappa>0$, $0<\varsigma<e^{-1}$ and $0<\theta <1$,  depending only on $n$, $\Omega$, $\tilde{g}$, $\alpha$, $\tilde{\beta}$,  such that
\begin{equation}\label{2.2}
\|c_1-c_2\|_\infty \le \kappa \Psi_{\varsigma,\theta}\left(\|\Lambda^1(0)-\Lambda^2(0)\|_{\mathscr{B}(H^{\frac{1}{2}}(\Gamma),H^{-\frac{1}{2}}(\Gamma))}\right).
\end{equation}
Using \eqref{mte1} with $j=0$ and the density of $H^{\frac{3}{2}}(\Gamma)$ in $H^{\frac{1}{2}}(\Gamma)$, we verify that 
\[
\|\Lambda^1(0)-\Lambda^2(0)\|_{\mathscr{B}(H^{\frac{1}{2}}(\Gamma),H^{-\frac{1}{2}}(\Gamma))}\le \mathbf{c}\delta(g_1,g_2).
\]
Upon modifying $\varsigma$, the inequality above and \eqref{2.2} yield 
\[
\|c_1-c_2\|_\infty \le \kappa \Psi_{\varsigma,\theta}\left(\delta(g_1,g_2)\right).
\]
The proof is complete.
\end{proof}

\begin{remark}
{\rm
In light of  \cite[Theorem 1.3]{DKLS}, the uniqueness result related to Theorem \ref{thm2.1} remains valid if we replace the assumption ``$(M',g')$ is a simple manifold'' with the injectivity of the ray transform on $(M',g')$.
}
\end{remark}

\section{BSD determines a hyperbolic DtN map}\label{s4}

Let $\tau >0$, $Q=\Omega \times (0,\tau)$, $\Sigma=\Gamma\times (0,\tau)$, and recall the following notations introduced in \cite{LM}, where $X=\Omega$ and $Y=Q$ or $X=\Gamma$ and $Y=\Sigma$,
\[
H^{r,s}(Y)=L^2((0,\tau),H^r(X))\cap H^s((0,\tau),L^2(X)),\quad r\ge 0,\; s\ge 0,
\]
with the convention that $H^0(\cdot)=L^2(\cdot)$. Define for all integer $j\ge 1$ 
\[
H_{0,}^j((0,\tau),H^{\frac{3}{2}}(\Gamma))=\left\{w\in H^j((0,\tau),H^{\frac{3}{2}}(\Gamma));\; \frac{d^\ell}{dt^\ell}w(0)=0,\; 0\le \ell\le j-1\right\},
\]
and
\[
\mathcal{H}^j((0,\tau),H^{\frac{3}{2}}(\Gamma))=H_{0,}^j((0,\tau),H^{\frac{3}{2}}(\Gamma))\cap H^{j+1}((0,\tau),H^{\frac{3}{2}}(\Gamma)).
\]

Consider the IBVP
\begin{equation}\label{w1.0.0}
\left\{
\begin{array}{ll}
(\partial_t^2-\Delta_g)w=F\quad \mathrm{in}\; Q,
\\
w(0,\cdot)=0,
\\
\partial_tw(0,\cdot)=0,
\\
w=h \quad \mathrm{on}\;  \Sigma .
\end{array}
\right.
\end{equation}

Let $g\in \mathbf{G}$ and $h\in \mathcal{H}^2((0,\tau),H^{\frac{3}{2}} (\Gamma))$. By \cite[Theorem 3.1 of Chapter 5]{LM}, the IBVP \eqref{w1.0.0} admits a unique solution $v\in H^{2,2}(Q)$ when $F=-(\partial_t^2-\Delta_g)Eh\in H^{0,1}(Q)$ and $h=0$. Therefore, $w=v+Eh\in H^{2,2}(Q)$ is the unique solution of the IBVP \eqref{w1.0.0}, when $F=0$. In the following, the solution of \eqref{w1.0.0} with $F=0$ and $h\in \mathcal{H}^2((0,\tau),H^{\frac{3}{2}} (\Gamma))$ will be denoted  $w^g(h)$.

Define the hyperbolic Dirichlet to Neumann map $\Pi^g$ associated with $g$ by
\[
\Pi^g:h\in \mathcal{H}^2((0,\tau),H^{\frac{3}{2}} (\Gamma))\mapsto \partial_{\nu_g}w\in L^2(\Sigma).
\]

The restriction of $\Pi^g$ to $\mathcal{H}^{2n+4}((0,\tau),H^{\frac{3}{2}} (\Gamma))$ will be denoted again $\Pi^g$.

In this section, $\mathbf{c}_\tau>0$ will denote a generic constant depending only on $n$, $\Omega$, $\alpha$, $\beta$ and $\tau$.

\begin{theorem}\label{mt2}
If $g_1,g_2\in \mathbf{G}_0$ satisfy $\delta_+(g_1,g_2)<\infty$, then
\begin{equation}\label{mte2}
\|\Pi^1-\Pi^2\|_{\mathscr{B}(\mathcal{H}^{2n+4}((0,\tau),H^{\frac{3}{2}} (\Gamma)), L^2(\Sigma))}\le \mathbf{c}_\tau \delta(g_1,g_2).
\end{equation}
\end{theorem}

Before proving Theorem \ref{mt2}, we establish two lemmas.

\begin{lemma}\label{lemwave}
Let $\ell \ge 1$ be an integer and $h\in \mathcal{H}^{2\ell+2}((0,\tau),H^{\frac{3}{2}} (\Gamma))$. Then the following formula holds.
\begin{equation}\label{for2}
\Pi^g(h)(\cdot ,t)=\sum_{j=0}^{\ell}(\Lambda^g)^{(j)}(0)((-\partial_t^2)^j  h(\cdot,t))+\partial_{\nu_g}r^g(h)(t),\quad t\in [0,\tau],
\end{equation}
where $r^g(h)\in H^{2,2}(Q)$ is the solution of the IBVP
\begin{equation}\label{w1.0}
\left\{
\begin{array}{ll}
(\partial_t^2-\Delta_g)r=\frac{(-1)^{\ell+1}}{\ell !}(u^g)^{(\ell)}(0)((-\partial_t^2)^{\ell+1}h(\cdot,t))\quad \mathrm{in}\; Q,
\\
r(0,\cdot)=0,\
\\
\partial_tr(0,\cdot)=0,
\\
r=0 \quad \mathrm{on}\;  \Sigma .
\end{array}
\right.
\end{equation}
\end{lemma}

\begin{proof}
For $t \in [0,\tau]$ and $\lambda\in \rho(-A^g)$, let
\[
W_0^g(h)(\lambda)(\cdot, t)=u^g(\lambda)(h(\cdot,t))
\]
and, for $1\le j\le \ell$, set
\[
W_j^g(h)(\lambda)(\cdot,t)=R^g(\lambda)(-\partial_t^2W_{j-1}^g(h)(\lambda)(\cdot, t)),\quad 1\le j\le \ell.
\]
Using the smoothness with respect to $\lambda$ and $t$, we have
\begin{align*}
W_j^g(h)(\lambda)(\cdot,t)&=R^g(\lambda)(-\partial_t^2R^g(\lambda)(-\partial_t^2W_{j-2}^g(h)(\lambda)(\cdot, t)))
\\
&=[R^g(\lambda)]^2((-\partial_t^2)^2W_{j-2}^g(h)(\lambda)(\cdot, t))
\\
&\ldots
\\
&=[R^g(\lambda)]^j((-\partial_t^2)^jW_0^g(h)(\lambda)(\cdot, t)).
\end{align*}
That is we have
\[
W_j^g(h)(\lambda)(\cdot,t)=[R^g(\lambda)]^j((-\partial_t^2)^ju^g(\lambda)(h(\cdot,t)))
\]
and hence
\begin{align*}
W_j^g(h)(\lambda)(\cdot,t)&=[R^g(\lambda)]^j(u^g(\lambda)((-\partial_t^2)^jh(\cdot,t)))
\\
&=-\sum_{k\ge 1}\frac{\langle(-\partial_t^2)^jh(\cdot,t)|\psi_k^g\rangle_g}{(\lambda_k^g+\lambda)^{j+1}}\phi_k^g.
\end{align*}
We derive from this formula the following one
\[
\frac{d}{d\lambda}(-\partial_t^2)W_j^g(h)(\lambda)(\cdot,t)=-(j+1)[R^g(\lambda)]^{j+1}(u^g(\lambda)((-\partial_t^2)^{j+1}h(\cdot,t))).
\]
Equivalently, we have
\[
W_j^g(h)(\lambda)(\cdot,t)=-\frac{1}{j}\frac{d}{d\lambda}(-\partial_t^2)W_{j-1}^g(h)(\lambda)(\cdot,t).
\]
Therefore, we obtain the following formula
\[
W_j^g(h)(\lambda)(\cdot,t)=\frac{(-1)^j}{j!}\frac{d^j}{d\lambda^j}(-\partial_t^2)^jW_0^g(h)(\lambda)(\cdot,t),
\]
from which we deduce that
\[
W_j^g(h)(\lambda)(\cdot,t)=\frac{(-1)^j}{j!}(u^g)^{(j)}(\lambda)((-\partial_t^2)^jh(\cdot,t)).
\]

Set
\begin{equation}\label{for0}
w_j^g(h)(\cdot,t):=W_j^g(h)(0)(\cdot,t)=\frac{(-1)^j}{j!}(u^g)^{(j)}(0)((-\partial_t^2)^jh(\cdot,t)).
\end{equation}

Note that, since $\partial_t^2w_\ell^g(h)\in H^1((0,\tau),H^2(\Omega))$, we have $r^g(h)\in H^{2,2}(Q)$.

Next,  we have $\Delta_g (u^g)^{(0)}(0)(h(\cdot,t))=0$ and from \eqref{spe}, we obtain
\[
 \Delta_g (u^g)^{(j)}(0)((-\partial_t^2)^jh(\cdot,t))=-j(u^g)^{(j-1)}(0)((-\partial_t^2)^jh(\cdot,t)),\; j\ge 1. 
\]
In other words, we have
\[
\Delta_gw_0^g(h)=0,\quad \Delta_g w^g_j(h)=\partial_t^2w_{j-1}^g,\; j\ge 1,
\]
and then
\[
(\partial_t^2-\Delta_g)\sum_{j=0}^\ell w_j^g(h)=\partial_t^2w_\ell^g(h).
\]
Hence 
\begin{equation}\label{for1}
\sum_{j=0}^\ell w_j^g(h)+r^g(h)=w^g(h).
\end{equation}
In light of \eqref{for0} and \eqref{for1}, we obtain
\[
\Pi^g(h)(\cdot ,t)=\sum_{j=0}^\ell (\Lambda^g)^{(j)}(0)((-\partial_t^2)^j  h(\cdot,t))+\partial_{\nu_g}r^g(h)(t).
\]
This is the expected formula.
\end{proof}

In the following, $r^g(h)$, $h\in \mathcal{H}^{2n+4}((0,\tau),H^{\frac{3}{2}} (\Gamma))$, will denote the solution of \eqref{w1.0} with $\ell=n+1$. 

\begin{lemma}\label{lemwave2}
We have for all $g\in \mathbf{G}$ and $h\in \mathcal{H}^{2n+4}((0,\tau),H^{\frac{3}{2}} (\Gamma))$ 
\begin{equation}\label{F3}
\partial_{\nu_g}r^g(h)= \sum_{k\ge 1}\int_0^t\frac{ \langle (-\partial_t^2)^{n+2}h(\cdot ,s)|\psi_k^g\rangle_g}{(\lambda_k^g)^{n+2}}s_k^g(t-s)\psi_k^g ds,
\end{equation}
where
\[
s_k^g(t)=\frac{\sin\left( \sqrt{\lambda_k^g}\, t\right)}{\sqrt{\lambda_k^g}},\quad t\in [0,\tau].
\]
\end{lemma}

\begin{proof}
We verify that
\[
r^g(h)(t)=\sum_{k\ge 1} r_k^g(t)\phi_k^g,
\]
where
\[
r_k^g(t)=\frac{(-1)^{n+2}}{(n+1)!}\int_0^t s_k^g(t-s)((u^g)^{(n+1)}(0)((-\partial_t^2)^{n+2}h(\cdot,s))|\phi_k^g)_gds.
\]
If $\tilde{h}:=(-\partial_t^2)^{n+2}h$, then 
\[
r_k^g(t)=\frac{(-1)^{n+2}}{(n+1)!}\int_0^t s_k^g(t-s)((u^g)^{(n+1)}(0)(\tilde{h}(\cdot,s))|\phi_k^g)_gds.
\]
But from \eqref{F1}, we have
\begin{align*}
(u^g)^{(n+1)}(0)(\tilde{h}(\cdot ,t))&=(-1)^{n+1}(n+1)!(R^g)^{n+1}(0)(u^g(0)(\tilde{h}(\cdot ,t))
\\
&=(-1)^{n+1}(n+1)!\sum_{\ell \ge 1}\frac{ (u^g(0)(\tilde{h}(\cdot ,t))|\phi_\ell^g)_g}{(\lambda_\ell^g)^{n+1}}\phi_\ell^g.
\end{align*}
Thus,
\[
r_k^g(t)=-\int_0^t\frac{ (u^g(0)(\tilde{h}(\cdot ,s))|\phi_k^g)_g}{(\lambda_k^g)^{n+1}}s_k^g(t-s)ds.
\]
On the other hand, an integration by parts yields
\[
(u^g(0)(\tilde{h}(\cdot ,s))|\phi_k^g)_g=-\frac{1}{\lambda_k^g}\langle \tilde{h}(\cdot ,s)|\psi_k^g\rangle_g.
\]
That is, we have
\[
r_k^g(t)=\int_0^t\frac{ \langle \tilde{h}(\cdot ,s)|\psi_k^g\rangle_g}{(\lambda_k^g)^{n+2}}s_k^g(t-s)ds
\]
and hence
\begin{equation}\label{F2}
r^g(h)= \sum_{k\ge 1}\int_0^t\frac{ \langle \tilde{h}(\cdot ,s)|\psi_k^g\rangle_g}{(\lambda_k^g)^{n+2}}s_k^g(t-s)\phi_k^gds.
\end{equation}
By Cauchy-Schwarz's inequality, we obtain
\begin{align*}
\left|\int_0^t \langle \tilde{h}(\cdot ,s)|\psi_k^g\rangle_g s_k^g(t-s)ds\right|^2&\le \mathbf{c}_\tau\frac{\|\tilde{h}\|_{L^2(\Sigma)}^2\|\psi_k^g\|_{L^2(\Gamma)}^2}{\lambda_k^g}
\\
&\le \mathbf{c}_\tau k^{\frac{3}{2n}}\|\tilde{h}\|_{L^2(\Sigma)}^2,
\end{align*}
which gives
\[
k^{\frac{7}{2n}}\left|\int_0^t\frac{ \langle \tilde{h}(\cdot ,s)|\psi_k^g\rangle_g}{(\lambda_k^g)^{n+2}}s_k^g(t-s)ds\right|^2\le \mathbf{c}_\tau k^{-\frac{2n-1}{n}}\|\tilde{h}\|_{L^2(\Sigma)}^2.
\]
Therefore, the series in the right hand side of \eqref{F2} converges in $H^{\frac{7}{4}}(\Omega)$ and then
\[
\partial_{\nu_g}r^g(h)= \sum_{k\ge 1}\int_0^t\frac{ \langle \tilde{h}(\cdot ,s)|\psi_k^g\rangle_g}{(\lambda_k^g)^{n+2}}s_k^g(t-s)\psi_k^gds.
\]
The proof is complete.
\end{proof}

\begin{proof}[Proof of Theorem \ref{mt2}]
Let $g_1,g_2\in \mathbf{G}_0$. We decompose $\partial_{\nu_{g_0}}r^1(h)-\partial_{\nu_{g_0}}r^2(h)$ into four terms:
\[
\partial_{\nu_{g_1}}r^1(h)-\partial_{\nu_{g_2}}r^2(h)=I_1+I_2+I_3+I_4,
\]
where
\begin{align*}
&I_1= \sum_{k\ge 1} \int_0^t\left(\frac{1}{(\lambda_k^1)^{n+2}}-\frac{1}{(\lambda_k^2)^{n+2}}\right)\langle \tilde{h}(\cdot ,s)|\psi_k^1\rangle_{g_0} s_k^1(t-s)\psi_k^1 ds,
\\
&I_2=\sum_{k\ge 1}\int_0^t\frac{ \langle \tilde{h}(\cdot ,s)|\psi_k^1-\psi_k^2\rangle_{g_0}}{(\lambda_k^2)^{n+2}}s_k^1(t-s)\psi_k^1 ds,
\\
&I_3=\sum_{k\ge 1}\int_0^t\frac{ \langle \tilde{h}(\cdot ,s)|\psi_k^2\rangle_{g_0}}{(\lambda_k^2)^{n+2}}(s_k^1(t-s)-s_k^2(t-s))\psi_k^1 ds,
\\
&I_4=\sum_{k\ge 1}\int_0^t\frac{ \langle \tilde{h}(\cdot ,s)|\psi_k^2\rangle_{g_0}}{(\lambda_k^2)^{n+2}}s_k^2(t-s)(\psi_k^1-\psi_k^2)ds.
\end{align*}
Proceeding similarly to the preceding sections, we show
\begin{equation}\label{W1}
\|\partial_{\nu_{g_0}}r^1(h)-\partial_{\nu_{g_0}}r^2(h)\|_{0,2}\le \mathbf{c}_\tau \delta(g_1,g_2)\|\tilde{h}\|_{0,2}.
\end{equation}
Combining \eqref{mte1}, \eqref{for2} and \eqref{W1}, we obtain \eqref{mte2}.
\end{proof}

\begin{proof}[Proof of Theorem \ref{thm2.3}] It is an immediate consequence of Theorems \ref{thm2.2} and \ref{mt2} and the fact that $\mathcal{H}^{2n+4}((0,\tau),H^{\frac{3}{2}} (\Gamma))$ is dense in $H_{0,}^1(\Sigma)$.
\end{proof}

\begin{remark}
{\rm
We emphasize that we cannot obtain a stability inequality by combining  Theorem \ref{mt2} and Theorem \ref{thm2.2} because we have only 
\[\left\| \Pi^1-\Pi^2  \right\|_{\mathscr{B}(\mathcal{H}^{2n+4}((0,\tau),H^{\frac{3}{2}} (\Gamma)), L^2(\Sigma))}\le \varrho \left\| \Pi^1-\Pi^2  \right\|_{\mathscr{B}(H_{0,}^1(\Sigma), L^2(\Sigma))},\] for some constant $\varrho>0$. 

On the hand, the proof of Theorem \ref{thm2.2} in \cite{SU} is not sufficiently detailed to extend the stability inequality with $\left\| \Pi^1-\Pi^2  \right\|_{\mathscr{B}(\mathcal{H}^{2n+4}((0,\tau),H^{\frac{3}{2}} (\Gamma)), L^2(\Sigma))}$ instead of $\left\| \Pi^1-\Pi^2  \right\|_{\mathscr{B}(H_{0,}^1(\Sigma), L^2(\Sigma))}$.
}
\end{remark}

\section{Extensions}\label{s5}

In the present section, we explain how to modify the preceding proofs to establish stability inequalities for the problem of determining the conductivity or the potential in an elliptic equation from the corresponding BSD.

\subsection{The conductivity} \label{sb5.1}

Throughout this section,  $\Omega$ is $C^{1,1}$ bounded domain of $\mathbb{R}^n$. The boundary of $\Omega$ will be denoted again by $\Gamma$. For $\alpha >1$ and $\beta >0$ fixed, let
\[
\mathcal{A}=\left\{ a\in C^{1,1}(\overline{\Omega});\; \alpha^{-1} \le a\le \alpha ,\; \|a\|_{C^{1,1}(\overline{\Omega})}\le \beta\right\}.
\]
For $a\in \mathcal{A}$,  consider the unbounded operator
\[
A^au=-\mathrm{div}(a\nabla u),\quad u\in D(A^a)=H_0^1(\Omega)\cap H^2(\Omega).
\]
The sequence of eigenvalues of $A^a$ will be denoted $(\lambda_k^a)$:
\[
0<\lambda_1^a\le \lambda_2^a\le \ldots \le \lambda_k^a\le \ldots, \quad \lim_{k\rightarrow \infty}\lambda_k^a=\infty.
\]
Let $(\phi_k^a)$ be an orthonormal basis of $L^2(\Omega)$ consisting of eigenfunctions such that
\[
\phi_k^a\in H_0^1(\Omega)\cap H^2(\Omega),\quad A_a\phi_k^a=\lambda_k^a\phi_k^a,\quad k\ge 1. 
\]
The following notation will be used hereinafter
\[
\psi_k^a=\partial_{\nu_a} \phi_k^a:= a\nabla \phi_k^a\cdot \nu,\quad k\ge 1,
\]
where, as before, $\nu$ denotes the outward unit normal vector field to $\Gamma$. 

As we have done before, to simplify notation, any quantity $X_k^{a_j}$ will be written $X_k^j$, $j=1,2$, $k\ge 1$.

As in Section \ref{s1},  fix $1\le p<\frac{2n}{2n-1}$ and $1\le q\le \frac{4n}{4n-1}$. For $a_1,a_2\in \mathcal{A}$, define
\begin{align*}
&\delta(a_1,a_2):=\|(\lambda_k^1-\lambda_k^2)\|_{\ell^p}+\|(\psi_k^1-\psi_k^2)\|_{\ell^q(L^2(\Gamma))},
\\
&\delta_0(a_1,a_2):=\|(\ell^{\frac{7}{4n}}k^{-\frac{1}{4n}}(\delta_{k\ell}-(\phi_k^2|\phi_\ell^1)_{\mathbf{e}}))\|_{\ell^1(\mathbb{N}^2)},
\\
&\delta_+(a_1,a_2):= \delta(g_1,g_2)+\delta_0(g_1,g_2),
%\\
%&\tilde{\delta}(a_1,a_2):=\|(\lambda_k^1-\lambda_k^2)\|_{\ell^{\tilde{p}}}+\|(\psi_k^1-\psi_k^2)\|_{\ell^{\tilde{q}}(L^2(\Gamma))}.
\end{align*}

Next, for $a\in \mathcal{A}$ and $\varphi \in H^{\frac{3}{2}}(\Gamma)$, $u^a(\varphi)\in H^2(\Omega)$  will denote the solution of the BVP
\[
-\mathrm{div}(a\nabla u)=0 \; \mathrm{in}\; \Omega ,\quad u{_{|\Gamma}}=\varphi,
\]
and
\[
v ^a(\varphi):=\partial_{\nu_a}u ^a(\varphi).
\]
By referring  to the a priori estimate in $H^2(\Omega)$, we verify that the mapping
\[
\Lambda^a:\varphi\in H^{\frac{3}{2}}(\Gamma)\mapsto v ^a(\varphi)\in L^2(\Gamma)
\]
defines a bounded operator.

It should be noted that the operator $A^a$ cannot be expressed in the form $A^g$ for any metric $g$. Therefore, Theorem \ref{mt1} is not applicable. However, we have a similar result  whose proof differs only slightly from that of Theorem \ref{mt1}. This result is stated below for the case $j=0$.

\begin{theorem}\label{mt1.1}
Let $a_1,a_2\in \mathcal{A}$ such that $\delta_+(a_1,a_2)<\infty$. Then $\Lambda^1-\Lambda^2$ extends to a bounded operator on $L^2(\Gamma)$, denoted again $\Lambda^1-\Lambda^2$, and we have 
\begin{equation}\label{mte1.1}
\|\Lambda^1-\Lambda^2\|_{\mathscr{B}(L^2(\Gamma))}\le \mathbf{c} \delta(a_1,a_2),
\end{equation}
where the constant $\mathbf{c} >0$  depends only on $n$, $\Omega$, $\alpha$ and $\beta$.
\end{theorem}

Let $\chi_I$ denotes the characteristic function of the interval $I$ and define
\[
\Psi_\varsigma=|\ln \delta|^{-\frac{2}{2+n}}\chi_{(0,\varsigma]}+\delta\chi_{(\varsigma,\infty)},\quad 0<\varsigma<1.
\]
We extend $\Psi_\varsigma$ by continuity at $\delta=0$ by setting $\Psi_\varsigma(0)=0$. 

We now use Theorem \ref{mt1.1} to prove the following result.

\begin{theorem}\label{thm1}
There exist two constants $\mathbf{c}>0$ and $0<\varsigma<1$, depending only on $n$, $\Omega$, $\alpha$ and $\beta$, such that for all $a_1,a_2\in \mathcal{A}$ verifying $\delta_+(a_1,a_2)<\infty$ we have
\[
\|a_1-a_2\|_{H^1(\Omega)}\le \mathbf{c}\Psi_\varsigma(\delta(a_1,a_2)).
\]
\end{theorem}

\begin{proof}
Let $a_1,a_2\in \mathcal{A}$ such that $\delta_+(a_1,a_2)<\infty$. By proceeding in a similar manner to the proof of Theorem \ref{thm2.1}, we obtain from Theorem \ref{mt1.1} 
\begin{equation}\label{equ1.1}
\|\Lambda^1-\Lambda^2\|_{\mathscr{B}(H^{\frac{1}{2}}(\Gamma),H^{-\frac{1}{2}}(\Gamma))}\le \mathbf{c}_1\delta (a_1,a_2),
\end{equation}
where the constant $\mathbf{c}_1 >0$  depends only on $n$, $\Omega$, $\alpha$ and $\beta$.

On the other hand, By \cite[(3.10) with $p=2n$]{Ch2} there exist three constants $\mathbf{c}_0>0$, $\gamma >0$ and $\rho_0>0$, depending only on $n$, $\Omega$, $\alpha$ and $\beta$, such that
\begin{equation}\label{x}
\mathbf{c}_0\|a_1-a_2\|_{H^1(\Omega)}\le \rho^{-\frac{2}{2+n}}+e^{\gamma\rho}\left[\|\Lambda^1-\Lambda^2\|_{\mathscr{B}(H^{\frac{1}{2}}(\Gamma),H^{-\frac{1}{2}}(\Gamma))}\right]^{\frac{1}{3}},\quad \rho\ge \rho_0.
\end{equation}
Putting together \eqref{equ1.1} and \eqref{x}, we get
\[
\mathbf{c}\|a_1-a_2\|_{H^1(\Omega)}\le \rho^{-\frac{2}{2+n}}+e^{\gamma \rho}\left[\delta(a_1,a_2)\right]^{\frac{1}{3}},\quad \rho\ge \rho_0.
\]
The expected inequality is then obtained by minimizing the right-hand side of the above inequality with respect to $\rho$.
\end{proof}

\subsection{The potential}\label{sb5.2}

Fix $\mathfrak{c}>0$ and define
\[
\mathbf{V}=\{V\in L^\infty(\Omega,\mathbb{R});\; 0\le V\le \mathfrak{c}\}.
\]

Let $g\in \mathbf{G}$ be fixed and, for $V\in \mathbf{V}$, define the unbounded operator $A^V$ by
\[
A^Vu=(-\Delta_g+V)u,\quad u\in D(A^V)=H_0^1(\Omega)\cap H^2(\Omega).
\]
As it is known, the spectrum of $A^V$ consists of a sequence $(\lambda_k^g)$:
\[
0<\lambda_1^V\le \lambda_2^V\le \ldots \le \lambda_k^V\le \ldots, \quad \lim_{k\rightarrow \infty}\lambda_k^V=\infty.
\]
In the following, $(\phi_k^V)$ will denote an orthonormal basis of $L^2(\Omega,dV_g)$ consisting of eigenfunctions such that
\[
\phi_k^V\in H_0^1(\Omega)\cap H^2(\Omega),\quad A^V\phi_k^V=\lambda_k^V\phi_k^V,\quad k\ge 1. 
\]
Set
\[
\psi_k^V:=\partial_{\nu_g} \phi_k^V,\quad k\ge 1.
\]

Modifying slightly the proof of \eqref{ee}, we obtain the following inequality
\begin{equation}\label{po1}
\vartheta^{-1}k^{\frac{2}{n}}\le \lambda_k^V\le \mathfrak{c}+\vartheta k^{\frac{2}{n}},\quad k\ge 1,\; V\in \mathbf{V},
\end{equation}
where $\vartheta$ is as in \eqref{ee}.

For $V\in \mathbf{V}$, $\lambda \in \rho(-A^V)$, $\varphi \in H^{\frac{3}{2}}(\Gamma)$, $u^V(\lambda)(\varphi)$ will denote the solution of the BVP
\[
(-\Delta_g+V+\lambda)u=0\; \mathrm{in}\; \Omega,\quad u_{|\Gamma}=\varphi.
\]
The family $(\Lambda^V(\lambda))_{\lambda \in \rho(-A^V)}$ of DtN maps associated with $V$ is defined as follows
\[
\Lambda^V(\lambda):\varphi \in H^{\frac{3}{2}}(\Gamma)\mapsto \partial_{\nu_g}u^V(\lambda)(\varphi)\in L^2(\Gamma).
\]

Unless otherwise stated, $\mathbf{c}>0$ will represent a generic constant depending only on $n$, $\Omega$, $g$, and $\mathfrak{c}$.

Similarly to Lemma \ref{lem2}, we verify that
\[
u^V(\lambda)(\varphi)=-\sum_{k\ge 1}\frac{\langle \varphi|\psi_k^V\rangle_g}{\lambda_k^V+\lambda}\phi_k^V.
\]
In particular, we have
\begin{equation}\label{po2}
\|u^V(\lambda)(\varphi)\|_{0,2}\le \|u^V(0)(\varphi)\|_{0,2}, \quad V\in \mathbf{V},\; \lambda \ge 0,\; \varphi \in H^{\frac{3}{2}}(\Gamma).
\end{equation}

Let $\tilde{E}\varphi \in H^1(\Omega)$ be the unique element of $H^1(\Omega)$ such that  $\|\tilde{E}\varphi \|_{1,2}=\|\varphi\|_{\frac{1}{2},2}$, where
\[
\|\varphi\|_{\frac{1}{2},2}:=\min\{\|h\|_{1,2},\; h\in H^1(\Omega),\; h_{|\Gamma}=\varphi\}
\]
is the quotient norm on $H^{\frac{1}{2}}(\Gamma)$ (which is equivalent to the usual norm of $H^{\frac{1}{2}}(\Gamma)$).

We verify that $u^V(0)(\varphi)=\tilde{E}\varphi +w$, where $w\in H_0^1(\Omega)$ is the unique weak solution of the BVP
\[
(-\Delta_g+V)w=-(-\Delta_g+V)\tilde{E}\varphi \; \mathrm{in}\; \Omega,\quad w_{|\Gamma}=0.
\]
Hence
\[
\int_\Omega(|\nabla_gw|_g^2+V|w|^2)dV_g=-\int_\Omega (\langle\nabla_g\tilde{E}\varphi,\nabla_g\overline{w}\rangle_g+V\tilde{E}\varphi \overline{w})dV_g,
\]
from which we obtain
\[
\|w\|_{1,2}\le \mathbf{c}\|\varphi\|_{\frac{1}{2},2}
\]
and then
\begin{equation}\label{po3}
\|u^V(0)(\varphi)\|_{1,2}\le \mathbf{c}\|\varphi\|_{\frac{1}{2},2}.
\end{equation}
Using \eqref{po3} in \eqref{po2} yields
\begin{equation}\label{po4}
\|u^V(\lambda)(\varphi)\|_{0,2}\le \mathbf{c}\|\varphi\|_{\frac{1}{2},2}, \quad V\in \mathbf{V},\; \lambda \ge 0,\; \varphi \in H^{\frac{3}{2}}(\Gamma).
\end{equation}

Henceforth, we use the notation $R^V(\lambda):=(A^V+\lambda)^{-1}$, $\lambda \in \rho(-A^V)$, and we recall that
\[
R^V(\lambda)f=\sum_{k\ge 1}\frac{(f|\phi_k^V)_g}{\lambda_k^V+\lambda}\phi_k,\quad f\in L^2(\Omega).
\]
In particular, we have
\begin{equation}\label{po5}
\|R^V(\lambda)f\|_{0,2,g}\le \lambda^{-1}\|f\|_{0,2,g},\quad \lambda \ge 1, \; f\in L^2(\Omega).
\end{equation}

We proceed as for Proposition \ref{prohol} to show that, for all  $V\in \mathbf{V}$ and $\varphi \in H^{\frac{3}{2}}(\Gamma)$, the mapping $\lambda \in \rho(-A^V)\mapsto u^V(\lambda)(\varphi)\in H^2(\Omega)$ is holomorphic.

As in previous sections, to simplify the notation, any quantity of the form $X_k^{V_j}$ will be written as $X_k^j$, where $j=1,2$ and $k \ge 1$.

\begin{lemma}\label{lempo1}
Let $V_1,V_2\in \mathbf{V}$, $\lambda \ge 1$ and $\varphi \in H^{\frac{3}{2}}(\Gamma)$. Then we have for all integer $j\ge 0$
\begin{equation}\label{po15}
\| \lambda^j[(\Lambda^1)^{(j)}(\lambda)(\varphi)-(\Lambda^2)^{(j)}(\lambda)(\varphi)]\|_{0,2}\le \mathbf{c}_j\lambda^{-\frac{1}{8}}\|\varphi\|_{\frac{1}{2},2}.
\end{equation}
In particular, it holds
\begin{equation}\label{po11}
\lim_{\lambda \rightarrow \infty}\| \lambda^j[(\Lambda^1)^{(j)}(\lambda)(\varphi)-(\Lambda^2)^{(j)}(\lambda)(\varphi)]\|_{0,2}=0.
\end{equation}
\end{lemma}

\begin{proof}
Clearly, we have
\begin{equation}\label{dif}
u(\lambda):= u^1(\lambda)(\varphi)-u^2(\lambda)(\varphi)=R^1(\lambda)((q_2-q_1)u^2(\lambda)(\varphi)).
\end{equation}
Combining \eqref{po4} and \eqref{po5}, we get
\begin{equation}\label{po6}
\|u(\lambda)\|_{0,2}\le \mathbf{c}\lambda^{-1}\|\varphi\|_{\frac{1}{2},2}.
\end{equation}
Using the usual a priori estimate in $H^2(\Omega)$, we verify that 
\[
\|u(\lambda)\|_{2,2}\le \mathbf{c}( \lambda \|u(\lambda)\|_{0,2}+\|(q_2-q_1)u^2(\lambda)(\varphi)\|_{0,2}),
\]
which, in light of \eqref{po4} and \eqref{po6}, implies
\begin{equation}\label{po7}
\|u(\lambda)\|_{2,2}\le \mathbf{c}\|\varphi\|_{\frac{1}{2},2}.
\end{equation}

As
\[
(-\Delta_g+V_2+\lambda )(u^2)^{(j)}(\varphi)=-j(u^2)^{(j-1)}(\varphi),\; \mathrm{in}\; \quad (u^2)^{(j)}(\varphi)_{|\Gamma}=0,\quad j\ge 0,
\]
we proceed as above to show the following inequality
\begin{equation}\label{po8}
\|(u^2)^{(j)}(\lambda)\|_{0,2}\le \mathbf{c}_j\lambda^{-j}\|\varphi\|_{\frac{1}{2},2},\quad j\ge 0.
\end{equation}
Then \eqref{po8} and an induction in $j$ allows us to get the following inequalities
\begin{align}
&\|u^{(j)}(\lambda)\|_{0,2}\le \mathbf{c}_j\lambda^{-(j+1)}\|\varphi\|_{\frac{1}{2},2},\quad j\ge 0,\label{po9}
\\
&\|u^{(j)}(\lambda)\|_{2,2}\le \mathbf{c}_j\lambda^{-j}\|\varphi\|_{\frac{1}{2},2},\quad j\ge 0.\label{po10}
\end{align}
By interpolation, \eqref{po9} and \eqref{po10} yield
\[
\|u^{(j)}(\lambda)\|_{\frac{7}{4},2}\le \mathbf{c}_j\lambda^{-(j+\frac{1}{8})}\|\varphi\|_{\frac{1}{2},2},\quad j\ge 0.
\]
This and the continuity of the trace operator $w\in H^{\frac{7}{4}}(\Omega)\mapsto \partial_{\nu_g} w\in L^2(\Gamma)$ gives \eqref{po15}.
\end{proof}

Let $p$ and $q$ be as in Section \ref{s1} and, for $V_1,V_2\in \mathbf{V}$, define
\[
\delta(V_1,V_2)=\|(\lambda_k^1-\lambda_k^2)\|_{\ell^p}+\|(\psi_k^1-\psi_k^2)\|_{\ell^q(L^2(\Gamma))}.
\]

Given \eqref{po11}, we mimic the proof of Theorem \ref{mt1} in order to obtain the following result.

\begin{theorem}\label{thmpo1}
Let $V_1,V_2\in \mathbf{V}$ such that $\delta (V_1,V_2)<\infty$ and $j\ge 0$ be an integer. Then for all $\varphi\in H^{\frac{3}{2}}(\Gamma)$ the following inequality holds
\begin{equation}\label{po12}
\|(\Lambda^1)^{(j)}(0)(\varphi)-(\Lambda^2)^{(j)}(0)(\varphi)\|_{0,2}\le \mathbf{c}_j \delta(V_1,V_2)\|\varphi\|_{0,2}.
\end{equation}
\end{theorem}

Fix $0<\eta<\frac{1}{2}$ and define
\[
\mathbf{V}_\eta=\{V\in L^\infty(\Omega,\mathbb{R}) \cap H^\eta(\Omega);\; 0\le V\le \mathfrak{c},\; \|V\|_{\eta,2}\le \mathfrak{c}\}.
\]
Also define
\[
\Phi_{\varsigma}=|\ln(\delta+| \ln \delta|^{-1}) |^{-\frac{\eta}{4}}\chi_{(0,\varsigma]}+\delta \chi_{(\varsigma,\infty)},\quad 0<\varsigma<e^{-1}.
\]
We extend $\Psi_{\varsigma}$ by continuity at $\delta=0$ by setting $\Psi_{\varsigma}(0)=0$.

The following theorem is obtained by combining Theorem \ref{thmpo1} with $j=0$ and \cite[Theorem 1.1]{CS}.

\begin{theorem}\label{thmpo2}
Assume that $(\Omega,g)$ is admissible with $\Omega\subset \Omega'\times \mathbb{R}$. For all $V_1,V_2\in \mathbf{V}_\eta$ such that $\delta(V_1,V_2)<\infty$ we have
\[
\|V_1-V_2\|_{L^2(\mathbb{R},H^{-1}(\Omega')}\le c\Phi_\varsigma\left(\delta(V_1,V_2)\right),
\]
where the constants $c>0$ and $\varsigma >0$ depend only on $n$, $\Omega'$, $\Omega$, $g$, $\mathfrak{c}$ and $\eta$.
\end{theorem}

Next, we show that by replacing $\delta(V_1,V_2)$ with $\delta(V_1,V_2)^\theta$ in \eqref{po12}, for some $0<\theta<1$, the assumption $\delta(V_1,V_2)<\infty$ can be removed, provided that $(\Omega,g)$ is simple.

Define for $V_1,V_2\in \mathbf{V}$
\[
\overline{\delta}(V_1,V_2)=\sum_{k\ge 1}\left[k^{-\frac{4k_0+1}{2n}}|\lambda_k^1-\lambda_k^2|+k^{-\frac{4k_0-3}{2n}}\|\psi_k^1-\psi_k^2\|_{0,2} \right],
\]
where $k_0$ is the smallest integer $>\frac{n+3}{4}$ chosen in such a way that $\overline{\delta}(V_1,V_2)<\infty$.

\begin{theorem}\label{thmpo3}
Let $V_1,V_2\in \mathbf{V}$  and $j\ge 0$ be an integer. Then for all $\varphi\in H^{\frac{3}{2}}(\Gamma)$ the following inequality holds
\begin{equation}\label{po12.1}
\|(\Lambda^1)^{(j)}(0)(\varphi)-(\Lambda^2)^{(j)}(0)(\varphi)\|_{0,2}\le \mathbf{c}_j \overline{\delta}(V_1,V_2)^\sigma \|\varphi\|_{\frac{1}{2},2},
\end{equation}
where $\sigma=\frac{1}{1+k_0}$.
\end{theorem}

\begin{proof}
Let $d_k^\ell$ be as in the proof of Theorem \ref{mt1} in which $j_0$ is replaced by $k_0$. That is, where $\lambda >0$ and $m\ge 0$ is an integer,
\[
d_k^\ell=\frac{(-1)^{m+1}(k_0+m)!}{(k_0-1)!}\int_0^1\frac{(1-t)^{k_0-1}\lambda^{k_0}}{(\lambda_k^\ell +(1-t)\lambda)^{k_0+m+1}}dt,\quad \ell=1,2.
\]
Hence 
\begin{equation}\label{po13}
|d_k^\ell |\le \mathbf{c}_m\lambda^{k_0}k^{-\frac{2(k_0+1)}{n}},\quad \ell=1,2.
\end{equation}
We verify that
\begin{equation}\label{po14}
|d_k^1-d_k^2 |\le \mathbf{c}_m\lambda^{k_0}k^{-\frac{2(k_0+2)}{n}}|\lambda_k^1-\lambda_k^2|.
\end{equation}
Using \eqref{po15}, \eqref{po13} and \eqref{po14},  by modifying slightly the proof of Theorem \ref{mt1}, we obtain that for all $\varphi\in H^{\frac{3}{2}}(\Gamma)$ 
\[
\|(\Lambda^1)^{(j)}(\varphi)-(\Lambda^2)^{(j)}(\varphi)\|_{0,2}\le \mathbf{c}_j\left( \lambda^{-\frac{1}{8}}+\lambda^{k_0}\overline{\delta}(V_1,V_2)\right)\|\varphi\|_{\frac{1}{2},2},\quad \lambda >0.
\]
The expected inequality follows by taking $\lambda$ so that $\lambda^{-\frac{1}{8}}=\lambda^{k_0}\overline{\delta}(V_1,V_2)$.
\end{proof}

For $V_1,V_2\in \mathbf{V}$, define
\[
\delta_\ast(V_1,V_2)=\sum_{k\ge 1}k^{-(2+\frac{5}{2n})}\left[|\lambda_k^1-\lambda_k^2|+\|\psi_k^1-\psi_k^2\|_{0,2}\right] \; (<\infty),
\]
and let
\[
\mathbf{V}_\ast=\{V\in C^\infty(\overline{\Omega},\mathbb{R}) ;\; 0\le V\le \mathfrak{c},\; \|V\|_{1,2}\le \mathfrak{c}\}.
\]

From now on, the notations we use are those of Section \ref{s4}. As in Section \ref{s4}, define the hyperbolic DtN map associated with $q\in \mathbf{V}_\ast$ by
\[
\Pi^V:h\in \mathcal{H}^2((0,\tau), H^{\frac{3}{2}}(\Gamma))\mapsto \partial_{\nu_g}w\in L^2(\Sigma),
\] 
where $w\in H^{2,2}(Q)$ is the solution of the IBVP
\[
\left\{
\begin{array}{ll}
(\partial_t^2-\Delta_g+V)w=0\quad \mathrm{in}\; Q,
\\
w(0,\cdot)=0,
\\
\partial_tw(0,\cdot)=0,
\\
w=h \quad \mathrm{on}\;  \Sigma .
\end{array}
\right.
\]

We proceed as in the proof of Theorem \ref{mt2} to obtain, by using inequality \eqref{po12.1},  that we have for all $h\in \mathcal{H}^{2n+4}((0,\tau), H^{\frac{3}{2}}(\Gamma))$ 
\begin{equation}\label{po16}
\|\Pi^1(h)-\Pi^2(h)\|_{L^2(\Sigma)}\le \mathbf{c}_\tau \left(\overline{\delta}(V_1,V_2)^\sigma+ \delta_\ast(V_1,V_2)\right)\|h\|_{H_{0,}^{2n+4}((0,\tau),H^{\frac{1}{2}}(\Gamma))}.
\end{equation}
Here and henceforth, $\mathbf{c}_\tau$ denotes a generic constant depending only on $n$, $\Omega$, $g$, $\mathfrak{c}$ and $\tau$.

On the other hand, under the assumption that $(\Omega,g)$ is simple, we have as a consequence of \cite[Proposition 5.8]{LQSY} 
\begin{equation}\label{po17}
\|V_1-V_2\|_{L^2(\Omega)}\le \|\Pi^1-\Pi^2\|_{\mathscr{B}(H_{0,}^{2n+4}((0,\tau),H^{\frac{1}{2}}(\Gamma)),L^2(\Sigma))}^\gamma,
\end{equation}
where the constant $\gamma\in (0,1)$ depend only on $n$, $\Omega$, $g$ and $\mathfrak{c}$.

In light of \eqref{po16} and \eqref{po17} and the density of $\mathcal{H}^{2n+4}((0,\tau),H^{\frac{1}{2}}(\Gamma))$ in $H_{0,}^{2n+4}((0,\tau),H^{\frac{1}{2}}(\Gamma))$, we can state the following result.

\begin{theorem}\label{thmpo4}
Assume that $(\Omega,g)$ is simple. For all $V_1,V_2\in \mathbf{V}_\ast$  we have
\[
\|V_1-V_2\|_{L^2(\Omega)}\le c\left(\overline{\delta}(V_1,V_2)^\sigma+ \delta_\ast(V_1,V_2)\right)^{\gamma},
\]
where $\gamma\in (0,1)$ is as in \eqref{po17}.
\end{theorem}

Theorem \ref{thmpo4} can be extended to the problem of determining both the magnetic and electric potentials in the magnetic Sch\"odinger operator from the corresponding BSD (see \cite[Theorem 1.1]{LQSY}).

A similar result to that of Theorem \ref{thmpo4} was obtained by the author in \cite{Ch24} for unbounded potentials. More precisely, when $V_1, V_2 \in L^{\frac{n}{2}}(\Omega, \mathbb{R})$ and $V_1 - V_2 \in H^2(\Omega)$ (\cite[Theorem 4.1]{Ch24}). In the case where $g = \mathbf{e}$, the condition $V_1 - V_2 \in H^2(\Omega)$ is replaced by $V_1 - V_2 \in L^n(\Omega)$ (\cite[Theorem 1.2]{Ch24}).

\end{document}